\documentclass[]{elsarticle}

\usepackage{lipsum}
\makeatletter
\def\ps@pprintTitle{%
 \let\@oddhead\@empty
 \let\@evenhead\@empty
 \def\@oddfoot{}%
 \let\@evenfoot\@oddfoot}
\makeatother

\usepackage{lineno,hyperref}
\modulolinenumbers[5]
\usepackage{amssymb,amsmath, amsthm,epsfig,multirow}
\usepackage{Tabbing}
\usepackage{bibentry}
\usepackage{float}

\usepackage{pdflscape}
\usepackage{afterpage}
\usepackage{capt-of}

\usepackage{lipsum}

\usepackage{rotating}

\newcommand\scalemath[2]{\scalebox{#1}{\mbox{\ensuremath{\displaystyle #2}}}}

\newcommand{\ignore}[1]{}

\newcommand{\eqm}{\begin{eqnarray}}
\newcommand{\enm}{\end{eqnarray}}
\newcommand{\eql}[1]{\begin{equation}\label{#1}}

\newcommand{\eqml}[1]{\eql{#1}\begin{array}{rcl}}
\newcommand{\enml}{\end{array}\end{equation}}

\newcommand{\eqmno}{\begin{eqnarray*}}
\newcommand{\enmno}{\end{eqnarray*}}

\newtheorem{theorem}{Theorem}

\newtheorem{proposition}{Proposition}
\newtheorem{lemma}{Lemma}
\newtheorem{corollary}{Corollary}

\def\cP{\mathcal{P}}
\def\cD{\mathcal{D}}
\def\cR{\mathcal{R}}

\def\bi{\begin{itemize}}
\def\ei{\end{itemize}}

\usepackage{xcolor}

\journal{}

\makeatletter
\def\@author#1{\g@addto@macro\elsauthors{\normalsize%
    \def\baselinestretch{1}%
    \upshape\authorsep#1\unskip\textsuperscript{%
      \ifx\@fnmark\@empty\else\unskip\sep\@fnmark\let\sep=,\fi
      \ifx\@corref\@empty\else\unskip\sep\@corref\let\sep=,\fi
      }%
    \def\authorsep{\unskip,\space}%
    \global\let\@fnmark\@empty
    \global\let\@corref\@empty  
    \global\let\sep\@empty}%
    \@eadauthor={#1}
}
\makeatother

\begin{document}

\begin{frontmatter}

\title{{\bf Cell-average WENO with progressive order of accuracy close to discontinuities with applications to signal processing}}

\author{Sergio Amat\fnref{fn1}}\ead{sergio.amat@upct.es}
\address{Department of Applied Mathematics and Statistics, Universidad Polit\'ecnica de Cartagena (Spain)}
\fntext[fn1]{This work was funded by project 20928/PI/18 (Proyecto financiado por la Comunidad Aut\'onoma de la Regi\'on de Murcia a trav\'es de la convocatoria de Ayudas a proyectos para el desarrollo de investigaci\'on cient\'ifica y t\'ecnica por grupos competitivos, incluida en el Programa Regional de Fomento de la Investigaci\'on Cient\'ifica y T\'ecnica (Plan de Actuaci\'on 2018) de la Fundaci\'on S\'eneca-Agencia de Ciencia y Tecnolog\'ia de la Regi\'on de Murcia) and through the national research project (MINECO/FEDER) and PID2019-108336GB-I00.}

\author{Juan Ruiz\corref{cor1}\fnref{fn1}} \ead{juan.ruiz@upct.es}
\address{Department of Applied Mathematics and Statistics, Universidad Polit\'ecnica de Cartagena (Spain)}

\author{Chi-Wang Shu \fnref{fn2}} \ead{Chi-Wang\_Shu@brown.edu}
\address{Division of Applied Mathematics, Brown University (USA)}
\fntext[fn2]{The author has been supported through the NSF grant DMS-2010107 and AFOSR grant FA9550-20-1-0055.}

\author{Dionisio F. Y\'a\~nez \fnref{fn3}} \ead{dionisio.yanez@uv.es}
\address{Department of Mathematics, Universidad de Valencia (Spain)}
\fntext[fn3]{The author has been supported through the Spanish MINECO project MTM2017-83942-P.}

\cortext[cor1]{Corresponding author}

\nobibliography{bibliografia_bibdesk_no_doi_url}

\begin{abstract}
In this paper we translate to the cell-average setting the algorithm for the point-value discretization presented in \bibentry{articulopaper}. This new strategy tries to improve the results of WENO-($2r-1$) algorithm close to the singularities, resulting in an optimal order of accuracy at these zones. The main idea is to modify the optimal weights so that they have a  nonlinear expression that depends on the position of the discontinuities. In this paper we study the application of the new algorithm to signal processing using Harten's multiresolution. Several numerical experiments are performed in order to confirm the theoretical results obtained.
\end{abstract}


\end{frontmatter}

\section{Introduction and review: Harten's multiresolution for cell-average setting}

Multiresolution and, in particular Harten's multiresolution (MR), has been used extensively for academic and industrial a\-pplications in the past years. Some examples are signal and image processing, compression or denoising (see, e.g. \cite{AD99,Har96}). There exists an interesting connection between wavelets theory, subdivision schemes and MR (see \cite{cohen}). While the approach presented in wavelets theory is based on functional analysis,
MR allows to analyze and study the problem from the point of view of function approximation and interpolation. The field is also connected to subdivision schemes, that are usually applied in computer aided design, and that can be modified to create a MR scheme under some theoretical conditions.

Harten's MR consists on four operators: let us suppose that $l$ is the level of resolution and that $\{x_j^l\}_{j=0}^{2^l}$  are equally spaced points of the interval $[0,1]$ with $x_j^l=j/2^l$, $h_l=1/2^l$. We interpret the data as the discretization of an integrable function, $L^1([0,1])$, through the discretization operator:
$$ \cD_l :L^1([0,1]) \to V^l\subset \mathbb{R}^l,$$
where $V^l$ is a discrete space. In the literature, several discretization operators can be found (see, e.g. \cite{AD99,Har96,10.2307/25663094,ardohar1}). In this paper, we consider our data as the averages of a function in the intervals $I_j^l=[(j-1)\cdot h^l,j\cdot h^l]$, i.e. each component of $\bar f^l$ is defined as:
\begin{equation}\label{discretiza}
\bar f^l_j:=(\cD_l f)_j=h^{-1}_l\int_{x_{j-1}^l}^{x_{j}^l} f(x) dx,\quad j=1,\hdots, 2^l.
\end{equation}
Once we have chosen the discretization, in order to move between two consecutive levels, two operators are defined: decimation and prediction. The decimation operator,
$$\cD^{l-1}_{l}:V^{l}\to V^{l-1}, $$
allows to go from a finer resolution to a coarser resolution and is a linear operator. In our case, we can define it using the linearity of the integral. Thus, for $j=1,\hdots,2^{l-1}$, we have that,
\begin{equation}\label{decimaciond}
(\cD^{l-1}_{l} \bar f^l)_j=\frac{1}{2} (\bar f^{l}_{2j-1}+\bar f^{l}_{2j})=\frac{2^l}{2} \left(\int_{x_{2j-2}^l}^{x_{2j-1}^l}f(x) dx + \int_{x_{2j-1}^l}^{x_{2j}^l}f(x)dx\right)={2^{l-1}} \int_{x_{j-1}^{l-1}}^{x_{j}^{l-1}}f(x)dx= \bar f^{l-1}_j.
\end{equation}
Now we have to design a prediction operator that allows to approximate the values at the level $l$ from data at the level $l-1$. We denote it as: $$\cP_{l-1}^l: V^{l-1} \to V^l.$$
In order to construct this operator, we need to impose a consistency property:
\begin{equation}\label{consistencia}
\cD^{l-1}_{l} \circ \cP_{l-1}^{l} = \mathbb{I}_{V^{l-1}},
\end{equation}
being $\mathbb{I}_{V^{l-1}}$ the identity function in the subspace $V^{l-1}$. This is a crucial property because it allows to recover all the information contained in the original data if we firstly apply the prediction operator and then the decimation operator.

Following the previous definitions, $\cP_{l-1}^l \bar f^{l-1}$ is an approximation of $\bar f^l$ and we can define the error vector as the difference between the original data and the approximation,
$$e^l_j=\bar f^l_j - (\cP_{l-1}^l \bar f^{l-1})_j, \quad j=1,\hdots, 2^l.$$
Then, using Eqs. \eqref{decimaciond} and \eqref{consistencia}, for any $j=1,\hdots,2^{l-1}$, we have that
\begin{equation}
\begin{split}
(\cD^{l-1}_{l} e^l)_j&=(\cD^{l-1}_{l} \bar f^l)_j - (\cD^{l-1}_{l} \cP_{l-1}^l \bar f^{l-1})_j,\\
\frac12(e^l_{2j}+e^l_{2j+1})&=\bar f^{l-1}_j-\bar f^{l-1}_j,\\
e^l_{2j}+e^l_{2j+1}&=0.
\end{split}
\end{equation}
We can define the vector of {\it details} as the non redundant information contained in the vector of errors. We can denote the vector of details by $d^l_{j}=e^l_{2j-1}$, $j=1,\hdots,2^{l-1}$. Therefore, if $L$ is the finest level of resolution, we can obtain a MR representation of $\bar f^L$ using the previously presented operators at the $L$ scales of resolution:
$$\bar f^L \equiv (\bar f^{0},d^1,\hdots,d^{L}),$$
The processing of the vectors of details at each scale of resolution allows to design several applications. For example, the truncation of the details using hard thresholding provides compression, the truncation using soft thresholding allows to reduce texture or noise, etc.

Hence, our principal objective is to construct a prediction operator which adequately approximates the original data. In the past years, some new prediction operators based on statistical tools have been designed (see, e.g. \cite{ARANDIGA20132474}). However, typically, linear and nonlinear interpolation techniques have been used (see, e.g. \cite{AD99,Har96,Amat2001273,sema}). At this point, it is necessary to design a reconstruction operator that allows to obtain an integrable function from the data at the level $l$, i.e.
$$\cR_l : V^l \to L^1([0,1]).$$
Taking this into account, we can define the prediction operator as the composition of the decimation operator and the reconstruction operator,
\begin{equation}\label{pred}
\cP_{l-1}^l:=\cD_l \circ \cR_{l-1}.
\end{equation}
Linear piecewise interpolation has been usually applied as the reconstruction operator $\cR_l$, (see, e.g. \cite{AD99,Har96,Amat2001273,sema}). However, this type of interpolation produces Gibbs phenomenon at discontinuity zones. In order to avoid this phenomenon, nonlinear techniques have been introduced in the past years. Essentially non oscillatory (ENO) method (see e.g. \cite{Har96,Abgrall2016xxi,Amat2001273,MR881365,Harten1987231}) has produced interesting results when it has been introduced as the reconstruction operator. In \cite{Liu}, the authors proposed the WENO (Weighted ENO) algorithm with the aim of obtaining a higher order of accuracy at smooth zones while keeping the adaption properties of ENO method close to the discontinuities. WENO method consists on a nonlinear convex combination of the interpolants constructed using the different stencils that ENO algorithm considers. For this, a set of nonlinear weights are designed by means of some values called smoothness indicators. These smoothness indicators are built using divided differences and different constructions can be found in the literature (see for example \cite{JiangShu, IS_nuestro}). The nonlinear weights constructed using these smoothness indicators assure that the stencils that cross a discontinuity have a very small contribution to the final approximation \cite{doi:10.1137/100791579, AMB, Castro20111766, Shu1998, doi:10.1137/070679065}.

In \cite{cellwenoamatruizshu} a new WENO-5 algorithm is designed for the cell averages in order to obtain maximum order at the intervals close to discontinuities. In this work we generalize that algorithm, constructing a WENO-($2r-1$)  algorithm for any odd number $r>1$, and perform some improvements oriented to obtain the same kind of order optimization close to discontinuities for stencils of any length. This paper can be considered as the second part of \cite{articulopaper}. In that paper, we presented the new algorithm interpreting that our data was discretized using the point values of a function.

We construct our algorithm showing general explicit formulas for any $r>1$. We design nonlinear weights that replace the optimal weights defined in the classical WENO-($2r-1$)  algorithm. These new weights are defined such that if there exists a discontinuity, then they tend to a power of $h_l$ and, in other case, they tend to a value that allows to obtain the maximum possible order of accuracy, taking into account all the data values of the stencil that are placed at smooth zones.

This paper is organized as follows: Section \ref{weno_sect} exposes how the classical WENO algorithm is constructed for the cell-average setting. Section \ref{sectionrelation} explains the relation between the construction in the point-values and the cell-average settings. Section \ref{WENO_cell} is dedicated to explain how to construct the new algorithm in the cell averages and to analyze theoretically its accuracy. In Section \ref{sectionrelation} we explain the relation between cell-average and point-value discretizations. Section \ref{num_exp} presents some experiments dedicated to analyze numerically the accuracy of the new algorithm through grid refinement analysis close to discontinuities. We also present examples of application of the new algorithm to the compression of univariate and bivariate functions. Finally, Section \ref{conclusions} presents the conclusions.

\section{Linear and nonlinear techniques. The classical WENO algorithm for the cell-averages}\label{weno_sect}
In this section we construct linear and nonlinear classical reconstruction operators and their corresponding prediction operators (see e.g. \cite{AD99} for more details).
We consider a fixed value $0\leq k\leq r-1$ and a stencil of intervals,
\begin{equation}\label{stencil}
S_k^{r}(j)=\{I_{j-r+k+1}^{l-1},\hdots, I^{l-1}_{j+k}\},
\end{equation}
with $I^{l-1}_{s}=[x^{l-1}_{s-1},x^{l-1}_{s}]$, $s=j-r+k+1,\hdots, j+k$, and we suppose a polynomial for each $j$, $p_k^{r}$, of degree $r-1$ such that:
$$\bar f^{l-1}_{s} = h^{-1}_{l-1}\int_{x^{l-1}_{s-1}}^{x^{l-1}_{s}} p_k^{r}(x) dx, \quad s=j-r+k+1,\hdots, j+k$$
then we define the reconstruction operator as:
$$\cR_{l-1}(\bar f^{l-1})(x)=p^{r}_k(x).$$
Finally, using Eq. \eqref{pred}, the prediction operator can be defined as:
\begin{equation}
(\mathcal{P}_{l-1}^l \bar f^{l-1})_j= (\mathcal{D}_l(\cR_{l-1}(\bar f^{l-1})))_j=h^{-1}_l\int_{x_{j-1}^l}^{x_{j}^l}\cR_{l-1}(\bar f^{l-1})(x)dx=h^{-1}_l\int_{x_{j-1}^l}^{x_{j}^l}p^{r}_k(x)dx.
\end{equation}
It is easy to check that the previous definition provides a linear operator. As the reconstruction operator is a polynomial, the previous prediction operator produces Gibbs phenomenon at the zones affected by discontinuities.

As an example, if $r=3$ we have three different linear prediction operators:
\begin{itemize}
\item For $k=0$ we obtain:
\begin{equation*}
(\mathcal{P}_{l-1}^l \bar f^{l-1})_{2j-1}= -\frac{1}{8} \bar f^{l-1}_{j-2} +\frac12 \bar f^{l-1}_{j-1} +\frac58\bar f_j^{l-1}, \qquad j=1,\hdots,2^{l-1}.
\end{equation*}
\item For $k=1$,
\begin{equation*}
(\mathcal{P}_{l-1}^l \bar f^{l-1})_{2j-1}= \frac{1}{8} \bar f^{l-1}_{j-1} +\bar f^{l-1}_{j} -\frac18\bar f_{j+1}^{l-1}, \qquad j=1,\hdots,2^{l-1}.
\end{equation*}
\item And for $k=2$,
\begin{equation*}
(\mathcal{P}_{l-1}^l \bar f^{l-1})_{2j-1}= \frac{11}{8}\bar  f^{l-1}_{j} -\frac12\bar f^{l-1}_{j+1} +\frac18\bar f_{j+2}^{l-1}, \qquad j=1,\hdots,2^{l-1}.
\end{equation*}
\end{itemize}
In all the three cases and the in the rest of paper, using Eq. \eqref{consistencia} we get that $(\mathcal{P}_{l-1}^l \bar f^{l-1})_{2j}=2\bar f^{l-1}_j-(\mathcal{P}_{l-1}^l \bar f^{l-1})_{2j-1}$, with $j=1,\hdots,2^{l-1}$. Then we only need to show the prediction operator for the cell $I^l_{2j-1}$ as the value of the cell $I^l_{2j}$ can be obtained through (\ref{decimaciond}).
When the function presents a discontinuity, linear methods are not appropriate due to the appearance of numerical effects close to the discontinuities. In this case, we can replace the linear interpolatory technique used in the prediction operator by a nonlinear method, such as WENO strategy. Next section is devoted to introduce the classical WENO algorithm.

\subsection{The classical WENO algorithm for the cell-averages}\label{weno_sect21}

We explain the WENO algorithm using the notation presented in \cite{AMB} and adapting it to the cell averages. As we have presented in previous section, if we consider the stencil $S_0^{2r-1}(j)$, we can construct the polynomial $p^{2r-1}_{0}$ of degree $2r-2$, which approximates $f$ in the cell averages at the intervals of the mentioned stencil. We can also consider the set of $r$ stencils of $r$ cells at the level $l-1$ which contain the interval $I^{l-1}_j=I^{l}_{2j-1}\cup I^{l}_{2j} $, i.e.
\begin{equation}\label{stencil}
S_k^{r}(j)=\{I_{j-r+k+1}^{l-1},\hdots, I^{l-1}_{j+k}\}, \quad k=0,\hdots,r-1.
\end{equation}
Let's denote by $p_k^{r}$, with $k=0,\hdots,r-1,$ the $r$ polynomials constructed over the previous stencils of $r$ cells. Then, we can combine these polynomials in order to obtain an interpolation of higher order. Thus, there exist values $\bar C^r_k$, with $k=0,\hdots,r-1,$ called optimal weights,  such that
\begin{equation}
\left(\cD_l \left(p^{2r-1}_{0} \right)\right)_{2j-1}=\left(\cD_l \left(\sum_{k=0}^{r-1}\bar C^r_k p^{r}_k \right)\right)_{2j-1}, \quad \text{with} \quad \sum_{k=0}^{r-1}\bar C^r_k =1.
\end{equation}
In the next section (Section \ref{sectionrelation}, Lemma \ref{lema1}) we will give an explicit expression for the optimal weights. As we will see in Section \ref{WENO_cell}, the main idea of the new algorithm is to replace the optimal weights, $\bar C^r_k$, by nonlinear weights, such that their value depends on the position of the discontinuity, if the discontinuity crosses the stencil $k$.
Thus, we will consider the following convex combination,
\begin{equation}\label{convex_combb}
\mathcal R_{l-1}(\bar f^{l-1})(x)=\sum_{k=0}^{r-1} \bar \omega_{k}^{r}(j) p_{k}^{r}(x),
\end{equation}
where $\bar \omega_{k}^{r}(j)\ge 0, \, k=0,\hdots, r-1$ and
$\sum_{k=0}^{r-1}\bar \omega_{k}^{r}(j)=1$. The prediction will be determined by:
\begin{equation}
(\mathcal{P}_{l-1}^l \bar f^{l-1})_{2j-1}=\left(\mathcal{D}_l (\mathcal R_{l-1}(\bar f^{l-1}))\right)_{2j-1}=\left(\mathcal{D}_l\left(\sum_{k=0}^{r-1} \bar \omega_{k}^{r}(j) p_{k}^{r}\right)\right)_{2j-1}.
\end{equation}
Classical WENO algorithm in the cell averages consists in designing the nonlinear weights with the aim of obtaining order $2r-1$ when there is no discontinuity crossing the biggest stencil $S^{2r-1}_0$, and of order $r$ in other case. In this sense WENO algorithm emulates the behavior of ENO algorithm. In \cite{Liu} the authors propose the already classic expression,
\begin{equation}\label{pesos}
\bar \omega_{k}^{r}(j)=\frac{\bar \alpha_{k}^{r}(j)}{\sum_{i=0}^{r-1}\bar \alpha_{i}^{r}(j)},\quad k=0,\cdots, r-1,\,\, \textrm{ where }\,\, \bar \alpha_{k}^{r}(j)=\frac{\bar C_{k}^{r}}{(\epsilon+L^r_k(j,\bar{f}))^t},
\end{equation}
where the values $\bar C_k^r$ are the optimal weights defined in Section \ref{sectionrelation}. The integer $t$ serves the purpose of maximizing the order of approximation. In \cite{JiangShu} the authors set $t=2$, in \cite{Liu} and  in \cite{articulopaper} the authors set $t=r$. In this paper we take $t=r$ in order to verify the ENO property and, at the same time, maximize the accuracy. The parameter $\epsilon$ takes positive and small values and its purpose is to avoid divisions by zero. In this article we set $\epsilon=10^{-16}$. Finally, $L_k^r(j,\bar{f})$ are the smoothness indicators for the cell-average values of $f$ on the stencil $S_{k}^r(j)$. In \cite{JiangShu} Jiang and Shu proposed to obtain the smoothness indicators using something similar to the total variation, but based in the $L^2$ norm, so that the result is smoother than the total variation. In Section \ref{sectionrelation} we will introduce the smoothness indicators presented in \cite{WENO_nuevo} adapted to the cell average setting.

Therefore, there are two principal ingredients to determine a WENO algorithm: the optimal weights, $\bar C^r_k$, $k=0,\hdots,r-1$, and the smoothness indicators $L^r_k(j,\bar f)$. In \cite{articulopaper}, in order to obtain progre\-ssive order of accuracy close to discontinuities, the optimal weights are replaced by nonlinear weights in the point-value discretization. It is clear that there exists a relation between the point-value and the cell-average discretizations (see \cite{AD99}). In the next section we review this relation and adapt the components of WENO of the new algorithm to the cell-average discretization.

\section{The relation between the point-value and the cell-average discretizations}\label{sectionrelation}
We can analyze the problem of approximating using the cell average discretization from the point of view of the point-value interpolation if we define the primitive function:
\begin{equation}\label{primitiva}
F(x)=\int_{0}^x f(t)dt,
\end{equation}
that allows to adapt the WENO algorithm presented in \cite{articulopaper} to the cell average setting in a natural way.

We denote the point-value discretization as $F^l_j:=(\cD^{pv}_l F)_j=F(x^l_{j})$. Then, setting $F_0^l=0$, we can define a linear application, $\Lambda_l:V^l\to V^l$, that allows to go from the
cell-average discretization to the point-values as:
\begin{equation}\label{pvtocell}
(\Lambda_l \bar f^{l})_j = h_l \sum_{i=1}^j \bar f^{l}_i=\int_0^{x_j^l}f(x)dx=F_j^l, \quad j=1,\hdots,2^l,
\end{equation}
and its inverse as:
\begin{equation}\label{celltopv}
(\Lambda^{-1}_l F^l)_j =h^{-1}_l(F^l_{j}-F^l_{j-1})=\bar f^l_j, \quad j=1,\hdots,2^l.
\end{equation}
We suppose that $\mathcal{I}_{l-1}(x;F^{l-1})$ is a polynomial interpolator of the primitive $F(x)$ at the nodes $\{x_{j-r+l-1}^{l-1},\hdots, x^{l-1}_{j+l-1}\}$, then we define the reconstruction for the cell averages as the derivative of the reconstruction for the primitive,
$$\cR_{l-1}(\bar f^{l-1})(x)=\frac{d}{dx}\mathcal{I}_{l-1}(x;\Lambda_{l-1}(\bar f^{l-1}))=\frac{d}{dx}\mathcal{I}_{l-1}(x;F^{l-1}).$$
Since $x_{2j-2}^l=x_{j-1}^{l-1}$, then by Eq. \eqref{pred}
\begin{equation}\label{diff2}
\begin{split}
(\cP_{l-1}^l \bar f^{l-1})_{2j-1}&=(\cD_l \circ \cR_{l-1}\bar  f^{l-1})_{2j-1} = h^{-1}_l\int_{x_{j-1}^l}^{x_{j}^l}\cR_{l-1}(\bar f^{l-1})(x)dx\\
&=h^{-1}_l\int_{x_{2j-2}^l}^{x_{2j}^l}\frac{d}{dx}\mathcal{I}_{l-1}(x;F^{l-1})dx\\
&=h^{-1}_l(\mathcal{I}_{l-1}(x_{2j-1}^l;F^{l-1}) - \mathcal{I}_{l-1}(x_{2j-2}^l;F^{l-1}))\\
&=h^{-1}_l(\mathcal{I}_{l-1}(x_{2j-1}^l;F^{l-1}) - F_{j-1}^{l-1})\\
&=h^{-1}_l(\mathcal{I}_{l-1}(x_{2j-1}^l;(\Lambda_{l-1} \bar f^{l-1})_{j-1}) - (\Lambda_{l-1} \bar f^{l-1})_{j-1}).
\end{split}
\end{equation}
The expression in (\ref{diff2}) is precisely the formula of the prediction operator for the cell average values.
Now, we consider the WENO algorithm for the point-value discretization of the primitive function $F$ defined in Eq. \eqref{primitiva}. Let $q^{2r}_{0}$ and $q^{r}_k$, with $k=0,\hdots,r-1$ be the polynomials which interpolate $F$ at the points $\{x^{l-1}_{j-r+1},\hdots,x^{l-1}_{j+r-1}\}$ and $\{x^{l-1}_{j-r+k},\hdots,x^{l-1}_{j+k}\}$ with $k=0,\hdots,r-1$. Now we can define:
\begin{equation}
\mathcal{I}_{l-1}(x_{2j-1}^{l};F^{l-1}):=\sum_{k=0}^{r-1} \omega^r_k(j) q^{r}_k(x_{2j-1}^{l}), \quad \text{with} \quad \sum_{l=0}^{r-1} \omega^r_k(j) =1,
\end{equation}
where $\omega_k^r(j)$ are linear or nonlinear weights. Now we can prove the following proposition:
\begin{proposition}\label{propcentral}
Let's consider $f\in L^1([0,1])$ and $F$ its primitive function defined in Eq. \eqref{primitiva}. Let $q^{r}_k$, with $k=0,\hdots,r-1$ be the polynomials which interpolate $F$ at points $\{x^{l-1}_{j-r+k},\hdots,x^{l-1}_{j+k}\}$ with $k=0,\hdots,r-1$, and $\omega^r_k(j)$, $k=0,\hdots,r-1$, the weights which satisfy that,
\begin{equation}
\mathcal{I}_{l-1}(x_{2j-1}^{l};F^{l-1}):=\sum_{k=0}^{r-1} \omega^r_k(j) q^{r}_k(x_{2j-1}^{l}), \quad \text{with} \quad \sum_{l=0}^{r-1} \omega^r_k(j) =1,
\end{equation}
and
$$F(x_{2j-1}^{l})-\mathcal{I}_{l-1}(x_{2j-1}^{l};F^{l-1})= O(h_{l-1}^{s+1}),$$
with $r\leq s < 2r$.
Let $p^{r}_k$ with $k=0,\hdots,r-1$, be the polynomials which approximate $f$ in the cell-average setting at the stencils $S_{k}^r$ with $k=0,\hdots,r-1$. In this case we have that,
$$\bar f^l_{2j-1}-\left(\mathcal{D}_l\left(\sum_{k=0}^{r-1} \omega_{k}^{r}(j) p_{k}^{r}\right)\right)_{2j-1}=O(h_{l-1}^s)$$
\end{proposition}

\begin{proof}
Let $q^r_k$ be the interpolatory polynomials of degree $r$ of $F$ at points $\{x_{j-r+k}^{l-1},\hdots, x_{j+k}^{l-1}\}$ with $k=0,\hdots,r-1$,
\begin{equation}
q^r_k(x_{n}^{l-1})=F_{n}^{l-1}, \quad n=j+k-r,\hdots,j+k,
\end{equation}
and let $p^{r}_{k}$ be the polynomials of degree $r-1$ which approximate $f$ in the cell-averages at $S_k^{r}(j)$, with $k=0,\hdots, r-1$, then it is easy to check that:
\begin{equation}
q^r_k(x)=\int_{x_{j-r+k}^{l-1}}^x p^{r}_{k}(t)dt+F_{j-r+k}^{l-1}, \quad k=0,\hdots,r-1,
\end{equation}
then
\begin{equation*}
\begin{split}
\mathcal{I}_{l-1}(x_{2j-1}^{l};F^{l-1})=&\sum_{k=0}^{r-1} \omega^r_k(j) q^{r}_k(x_{2j-1}^{l})=\sum_{k=0}^{r-1} \left(\omega^r_k(j) \left(\int_{x_{j-r+k}^{l-1}}^{x_{2j-1}^{l}} p^{r}_{k}(t)dt+F_{j-r+k}^{l-1}\right)\right)\\
=&\sum_{k=0}^{r-1} \left(\omega^r_k(j) \left(\int_{x_{j-1}^{l-1}}^{x_{2j-1}^{l}} p^{r}_{k}(t)dt+\int_{x_{j-r+k}^{l-1}}^{x_{j-1}^{l-1}} p^{r}_{k}(t)dt+F_{j-r+k}^{l-1}\right)\right)\\
=&\sum_{k=0}^{r-1} \left(\omega^r_k(j) \left(\int_{x_{j-1}^{l-1}}^{x_{2j-1}^{l}} p^{r}_{k}(t)dt+F_{j-1}^{l-1}\right)\right)\\
=&\sum_{k=0}^{r-1} \left(\omega^r_k(j) \left(\int_{x_{j-1}^{l-1}}^{x_{2j-1}^{l}} p^{r}_{k}(t)dt\right)\right)+F_{j-1}^{l-1}= \int_{x_{j-1}^{l-1}}^{x_{2j-1}^{l}}\left(\sum_{k=0}^{r-1}\omega^r_k(j)  p^{r}_{k}(t)dt\right)+F_{j-1}^{l-1}\\
=& h_l\left(\mathcal{D}_l\left(\sum_{k=0}^{r-1} \omega_{k}^{r}(j) p_{k}^{r}\right)\right)_{2j-1}+F_{j-1}^{l-1}\\
\end{split}
\end{equation*}
thus,
\begin{equation}\label{eq3}
\left(\mathcal{D}_l\left(\sum_{k=0}^{r-1} \omega_{k}^{r}(j) p_{k}^{r}\right)\right)_{2j-1}= h_l^{-1}(\mathcal{I}(x_{2j-1}^{l};F^{l-1})-F_{j-1}^{l-1}),
\end{equation}
and by Eq. \eqref{celltopv}:
\begin{equation*}
\bar f^{l}_{2j-1}=h_l^{-1}(F^{l}_{2j-1}-F^l_{2j-2})=h_l^{-1}(F^{l}_{2j-1}-F^{l-1}_{j-1}).
\end{equation*}
Thus, by \eqref{celltopv},  \eqref{eq3} and by hypothesis, we have:
\begin{equation*}
\begin{split}
\bar f^l_{2j-1}-\left(\mathcal{D}_l\left(\sum_{k=0}^{r-1} \omega_{k}^{r}(j) p_{k}^{r}\right)\right)_{2j-1}&=h^{-1}_l\left(F^l_{2j-1}-F^{l-1}_{j-1}- \mathcal{I}_{l-1}(x_{2j-1}^{l};F^{l-1})+F^{l-1}_{j-1}\right) \\
&=\frac{2O(h_{l-1}^{s+1})}{h_{l-1}}=O(h_{l-1}^{s}).
\end{split}
\end{equation*}
\end{proof}
This proposition allows us to define the new WENO method using the algorithm for the point-value discretization designed in \cite{articulopaper}. The weights $\omega_k^r(j)$ are exactly the same in both discretizations. If  $\bar\omega_k^r(j)$ are the weights defined for the cell-average setting and $\omega_k^r(j)$ the ones for the point-value setting, then,
$$ \bar\omega_k^r(j) = \omega_k^r(j), \quad k=0,\hdots,r-1.$$

Finally, we only have to determine the relation between the smoothness indicators for $L^r_k(j,F)$ and $L^r_k(j,\bar f)$.

In \cite{cellwenoamatruizshu,WENO_nuevo, generalizacion} the authors proposed to use the smoothness indicators given by the expression
\begin{equation}\label{si_nuestro2}
L_k^r(j,F)=\sum_{i=2}^{r} h_{l-1}^{2i-1}\int_{x^{l-1}_{j-1}}^{x^{l-1}_{j}}\left(\frac{d^i}{dx^i}q_{k}^{r}(x)\right)^2 dx,
\end{equation}
where $q_k^r$ are the interpolatory polynomials of $F$ at points $\{x_{j-r+k}^{l-1},\hdots, x_{j+k}^{l-1}\}$ with $k=0,\hdots,r-1$. In terms of the
cell-averages, the formula would be,
\begin{equation}\label{si_nuestro3}
L_k^{r}(j,\bar{f})=\sum_{i=1}^{r-1} h_{l-1}^{2i-1}\int_{x^{l-1}_{j-1}}^{x^{l-1}_{j}}\left(\frac{d^i}{dx^i}p_{k}^{r}(x)\right)^2 dx,
\end{equation}
where $p_{k}^{r}$ are the polynomials that approximate $f$ in the cell-averages at the stencils $S^r_k$, $k=0,\hdots,r-1$.
In \cite{cellwenoamatruizshu,WENO_nuevo, generalizacion} it is proved that the smoothness indicators in the cell-averages
(\ref{si_nuestro3}) and the ones for the point-value setting (\ref{si_nuestro2}) proposed in \cite{WENO_nuevo} are related through the expression,
\begin{equation}\label{relacion}
L_k^r(j,F)=h_{l-1}^2 L_k^r(j,\bar{f}).
\end{equation}
It is easy to check that the smoothness indicators can be obtained through (\ref{si_nuestro3}) using the previous polynomials. The expression for the smoothness indicators for $r=3$ cells can be expressed in terms of finite differences as,
\begin{equation}\label{IS4}
\begin{aligned}
L_{0}^2(j,\bar{f})&=\frac{10}{3}(\delta^2_{j-2})^2+3 \delta^2_{j-2} \delta_{j-2}+(\delta_{j-2})^2,\\
L_{1}^2(j,\bar{f})&=\frac{4}{3}(\delta^2_{j-1})^2+ \delta^2_{j-1} \delta_{j-1}+(\delta_{j-1})^2,\\
L_{2}^2(j,\bar{f})&=\frac{4}{3}(\delta^2_{j})^2-3 \delta^2_{j} \delta_{j}+(\delta_{j})^2,
\end{aligned}
\end{equation}
with $\delta_{j}= \bar{f}_{j+1}-\bar{f}_{j}$ and $\delta_{j}^2= \bar{f}_{j}-2\bar{f}_{j+1}+\bar{f}_{j+2}$.
The smoothness indicators for $r=4$ cells obtained using the same process are,

\begin{equation}\label{IS5}
\begin{aligned}
\tilde L_{0}^3&={\frac {27}{2}}\,\delta^3_{j-3}\,\delta^2_{j-3}+\frac{11}{3}\,\delta^3_{j-3}\,\delta_{j-3}+5\,\delta^2_{j-3}\,\delta_{j-3}+{\frac {2107}{240}}\,(\delta^3_{j-3})^{2}+{\frac {22}{3}}\,(\delta^2_{j-3})^{2}+(\delta_{j-3})^{2},\\
\tilde L_{1}^3&={\frac {547}{240}}(\delta^3_{j-2})^{2}+\frac{10}{3}\,(\delta^2_{j-2})^{2}+(\delta_{j-2})^{2}+{\frac {19}{6}}\,\delta^3_{j-2}\,\delta^2_{j-2}+\frac{2}{3}\,\delta^3_{j-2}\,\delta_{j-2}+3\,\delta^2_{j-2}\,\delta_{j-2},\\
\tilde L_{2}^3&={\frac {89}{80}}\,(\delta^3_{j-1})^{2}+\frac{4}{3}\,(\delta^2_{j-1})^{2}+(\delta_{j-1})^{2}-\frac{1}{6}\,\delta^3_{j-1}\,\delta^2_{j-1}-\frac{1}{3}\,\delta^3_{j-1}\,\delta_{j-1}+\delta^2_{j-1}\,\delta_{j-1},\\
\tilde L_{3}^3&={\frac {547}{240}}\,(\delta^3_{j})^{2}+\frac{4}{3}\,(\delta^2_{j})^{2}+(\delta_{j})^{2}-\frac{5}{2}\,\delta^3_{j}\,\delta^2_{j}+\frac{2}{3}\,\delta^3_{j}\,\delta_{j}-\delta^2_{j}\,\delta_{j},
\end{aligned}
\end{equation}
with $\delta_{j}^3={ -\bar{f}_{j}}+3\,{ \bar{f}_{j+1}}-3\,{ \bar{f}_{j+2}}+{ \bar{f}_{j+3}}$. For $r=5$ we obtain the following smoothness indicators
\begin{equation}\label{IS6}
\begin{aligned}
\tilde L_{0}^4&={\frac {53959}{2520}}\,(\delta^4_{j-4})^{2}+{\frac {2369}{80}}\,(\delta^3_{j-4})^{2}+{\frac {40}{3}}\,(\delta^2_{j-4})^{2}+(\delta_{j-4})^{2}+{\frac {10183}{240}}\,\delta^4_{j-4}\,\delta^3_{j-4}+{\frac {209}{10}}\,\delta^4_{j-4}\,\delta^2_{j-4}\\
&+{\frac {25}{6}}\,\delta^4_{j-4}\,\delta_{j-4}+{\frac {221}{6}}\,\delta^2_{j-4}\,\delta^3_{j-4}+{\frac {26}{3}}\,\delta_{j-4}\,\delta^3_{j-4}+7\,\delta_{j-4}\,\delta^2_{j-4},\\
\tilde L_{1}^4&={\frac {651}{80}}\,\delta^4_{j-3}\,\delta^3_{j-3}+{\frac {97}{30}}\,\delta^4_{j-3}\,\delta^2_{j-3}+\frac{1}{2}\,\delta^4_{j-3}\,\delta_{j-3}+{\frac {27}{2}}\,\delta^2_{j-3}\,{\it
d42}+\frac{11}{3}\,\delta_{j-3}\,\delta^3_{j-3}+5\,\delta_{j-3}\,\delta^2_{j-3}\\
&+{\frac {11329}{2520}}\,(\delta^4_{j-3})^{2}+{\frac {2107}{240}}\,(\delta^3_{j-3})^{2}+{\frac {22}{3}}\,(\delta^2_{j-3})^{2}+(\delta_{j-3})^{2},\\
\tilde L_{2}^4&={\frac {203}{240}}\,\delta^4_{j-2}\,\delta^3_{j-2}-{\frac {13}{30}}\,\delta^4_{j-2}\,\delta^2_{j-2}+{\frac {1727}{1260}}\,(\delta^4_{j-2})^{2}+{\frac {547}{240}}\,(\delta^3_{j-2})^{2}+\frac{10}{3}\,(\delta^2_{j-2})^{2}+(\delta_{j-2})^{2}-\frac{1}{6}\,\delta^4_{j-2}\,\delta_{j-2}\\
&+{\frac {19}{6}}\,\delta^2_{j-2}\,\delta^3_{j-2}+\frac{2}{3}\,\delta_{j-2}\,\delta^3_{j-2}+3\,\delta_{j-2}\,\delta^2_{j-2},\\
\tilde L_{3}^4&=-{\frac {89}{80}}\,\delta^4_{j-1}\,\delta^3_{j-1}-\frac{1}{10}\,\delta^4_{j-1}\,\delta^2_{j-1}+\frac{1}{6}\,\delta^4_{j-1}\,\delta_{j-1}-\frac{1}{6}\,\delta^2_{j-1}\,\delta^3_{j-1}-\frac{1}{3}\,\delta_{j-1}\,\delta^3_{j-1}+\delta_{j-1}\,\delta^2_{j-1}\\
&+{\frac {1727}{1260}}\,(\delta^4_{j-1})^{2}+{\frac {89}{80}}\,(\delta^3_{j-1})^{2}+\frac{4}{3}\,(\delta^2_{j-1})^{2}+(\delta_{j-1})^{2},\\
\tilde L_{4}^4&={\frac {11329}{2520}}\,(\delta^4_{j})^{2}+{\frac {547}{240}}\,(\delta^3_{j})^{2}+\frac{4}{3}\,(\delta^2_{j})^{2}+(\delta_{j})^{2}-\frac{1}{2}\,\delta^4_{j}\,\delta_{j}-\frac{5}{2}\,\delta^2_{j}\,\delta^3_{j}+\frac{2}{3}\,\delta_{j}\,\delta^3_{j}-\delta_{j}\,\delta^2_{j}-{\frac {1297}{240}}\,\delta^4_{j}\,\delta^3_{j}\\
&+{\frac {67}{30}}\,\delta^4_{j}\,\delta^2_{j},
\end{aligned}
\end{equation}
with $\delta_{j}^4={ \bar{f}_{j}}-4\,{ \bar{f}_{j+1}}+6\,{ \bar{f}_{j+2}}-4{ \bar{f}_{j+3}}+{ \bar{f}_{j+4}}$.

For the sake of completeness, we obtain the optimal weights mentioned in the previous section using the following proposition proved in \cite{AMB}.

\begin{proposition}{}\label{Propox1}
Let $F$ be the primitive function of $f$ defined in Eq. \eqref{primitiva}, $q^{2r}_{0}$ and $q^{r}_k$ the polynomials which interpolate $F$ at points $\{x^{l-1}_{j-r},\hdots,x^{l-1}_{j+r-1}\}$ and $\{x^{l-1}_{j-r+k},\hdots,x^{l-1}_{j+k}\}$ and $\bar C^r_k$  with $k=0,\hdots,r-1$, the values which satisfy:
$$q^{2r-1}_{0}(x_{2j-1}^{l})=\sum_{k=0}^{r-1}\bar C^r_k q^{r}_k(x_{2j-1}^{l}), \quad \text{with} \quad \sum_{l=0}^{r-1}\bar C^r_k =1,$$
then
\begin{equation}\label{equationpesosoptimos}
\bar C^r_k=\frac{1}{2^{2r-1}} {2r \choose 2k+1}, \quad k=0,\dots,r-1.
\end{equation}
\end{proposition}

With the following lemma we will prove that the optimal weights are the same in the cell-average and the point-value discretizations. The proof is similar to the one presented for Proposition \ref{propcentral}.
\begin{lemma}\label{lema1}
Let us consider $f\in L^1([0,1])$ and let $p^{2r}_{0}$ and $p^{r}_k$ with $k=0,\hdots,r-1$, be the polynomials which approximate $f$ in the cell-averages at the stencils $S_{0}^{2r-1}$ and $S_{k}^r$ respectively. Then, $\bar C^r_k$  with $k=0,\hdots,r-1$, are the values which satisfy,
$$\left(\cD_l \left(p^{2r-1}_{0}\right)\right)_{2j-1}=\left(\cD_l \left(\sum_{k=0}^{r-1}\bar C^r_k p^{r}_k\right)\right)_{2j-1}, \quad \text{with} \quad \sum_{k=0}^{r-1}\bar C^r_k =1,$$
being $\cD_l$ the operator defined in Eq. \eqref{discretiza}
then
\begin{equation*}
\bar C^r_k=\frac{1}{2^{2r-1}} {2r \choose 2k+1}, \quad k=0,\dots,r-1.
\end{equation*}
\end{lemma}

In \cite{articulopaper,cellwenoamatruizshu,WENO_nuevo,generalizacion}, new weights are constructed with the aim of obtaining progressive order of accuracy close to discontinuities (if they are far enough from each other). We review the construction of the weights presented in \cite{articulopaper} in Section \ref{WENO_cell}. The strategy consists in replacing the linear optimal weights $\bar C_{k}^{r}$ in \eqref{pesos} by nonlinear weights.

\section{Review of the new WENO algorithm and its properties}\label{WENO_cell}

In this section, we review the new algorithm presented in \cite{articulopaper} for any $r$. In \eqref{pesos} we replace the $\bar \alpha_{k}^{r}(j)$ by $\tilde \alpha_{k}^{r}(j)$,
\begin{equation}\label{pesosr1finales}
\tilde \omega_{k}^{r}(j)=\frac{\tilde \alpha_{k}^{r}(j)}{\sum_{i=0}^{r-1}\tilde \alpha_{i}^{r}(j)},\quad k=0,\cdots, r-1,\,\, \textrm{ where }\,\, \tilde \alpha_{k}^{r}(j)=\frac{\tilde C_{k}^{r}}{(\epsilon+L^r_k(j,\bar{f}))^t},
\end{equation}
where we have replaced the $\bar C_k^{r}$ by,
\begin{equation}\label{eqsuperimp}
\begin{split}
(\tilde{C}_0^r,\hdots,\tilde{C}_{r-1}^r)=\sum_{j_0=0}^1 \tilde{\omega}^{2r-2}_{0,j_0}\left(\sum_{j_1=j_0}^{j_0+1}\tilde{\omega}_{j_0,j_1}^{2r-3}\left(\sum_{j_2=j_1}^{j_1+1}\tilde{\omega}_{j_1,j_2}^{2r-4}
\left(\dots\left(\sum_{j_{r-2}=j_{r-3}}^{j_{r-3}+1} \tilde{\omega}^{r+1}_{j_{r-3},j_{r-2}}{\bf C_{j_{r-2}}^{r+1}} \right)\dots\right)\right)\right),
\end{split}
\end{equation}
with,
\begin{equation}\label{nl_op_w_r}
\begin{split}
&{\bf C_{0}^{r+1}}=\left(C_{0,0}^r, C_{0,1}^r, 0,0, \hdots,0\right)=\left(\frac{3}{2(r+1)}, \frac{2r-1}{2(r+1)},0, 0, \dots,0\right),\\
&{\bf C_{1}^{r+1}}=\left(0,C_{1,1}^r, C_{1,2}^r,0, \hdots, 0\right)=\left(0, \frac{5}{2(r+1)},\frac{2r-3}{2(r+1)},0,\dots,0\right),\\
 &\qquad \qquad\qquad\vdots\\
&{\bf C_{r-3}^{r+1}}=\left(0,\hdots,0,C_{r-3,r-3}^r, C_{r-3,r-2}^r, 0\right)=\left(0,\hdots,0,\frac{2r-3}{2(r+1)},\frac{5}{2(r+1)}, 0\right),\\
&{\bf C_{r-2}^{r+1}}=\left(0,\hdots,0, 0,C_{r-2,r-2}^r, C_{r-2,r-1}^r\right)=\left(0,\hdots,0, 0, \frac{2r-1}{2(r+1)},\frac{3}{2(r+1)}\right).\\
\end{split}
\end{equation}
and
\begin{equation}\label{pesosr}
\begin{split}
&\tilde{\omega}^n_{k,k_1}=\frac{\tilde{\alpha}_{k,k_1}^n}{\tilde{\alpha}_{k,k}^n+\tilde{\alpha}_{k,k+1}^n},\quad k_1=k,\,\, k+1,\\
&\tilde{\alpha}_{k,k_1}^n=\frac{C_{k,k_1}^n}{(\epsilon+\tilde L^n_{k,k_1}(j,F))^t}, \quad k_1=k,\,\,k+1,\\
\end{split}
\end{equation}
being
\begin{equation}\label{pesostodos}
C^n_{k,k}=\frac{2(n-r+k+1)+1}{2(n+1)}, \quad C^n_{k,k+1}=1-C^n_{k,k}=\frac{2(r-k)-1}{2(n+1)}.
\end{equation}
The smoothness indicators $\tilde L^n_{k,k_1}(j,F)$ are defined as,
\begin{equation}\label{equationindicadores}
\begin{split}
&\tilde L_{k,k}^{n}(j,F)=L_k^r(j,F), \quad k=0,\dots,(2r-2)-n,\\
&\tilde L_{k,k+1}^{n}(j,F)=L_{n-(r-1)+k}^r(j,F), \quad k=0,\dots,(2r-2)-n,
\end{split}
\end{equation}
where $L^r_k(j,F)$, with $k=0,\hdots,r-1$ are the smoothness indicators shown in Eq. \eqref{si_nuestro2}. With these changes, in \cite{articulopaper} its is proved the following theorem:

\begin{theorem}\label{teo1}
Let's consider $1< l_0 \leq r-1$, $\tilde \omega_k^r(j)$ defined in Eq. \eqref{pesosr1finales} and $q^r_k$ the interpolatory polynomials for $F$ at the points $\{x_{j-r+k}^{l-1},\hdots, x_{j+k}^{l-1}\}$ with $k=0,\hdots,r-1$.
 If $f$ is smooth in $[x^{l-1}_{j-r},x^{l-1}_{j+r-1}]\setminus \Omega$ and $f$ has a discontinuity at $\Omega$ then
\begin{equation}
\sum_{k=0}^{r-1}\tilde \omega^r_k(j) q^r_k(x^l_{2j-1})-F(x^l_{2j-1})=\left\{
                                                  \begin{array}{ll}
                                                    O(h_{l-1}^{2r}), & \hbox{if $\,\,\Omega=\emptyset$;} \\
O(h_{l-1}^{r+l_0}), & \hbox{if   $\,\,\Omega=[x^{l-1}_{j+l_0-1},x^{l-1}_{j+l_0}]$; }\\
                                                  \end{array}
                                                \right.
\end{equation}
\end{theorem}
Analogously, if the discontinuity is placed in the interval $\Omega=[x_{j-l_0},x_{j-l_0-1}]$, with $1<l_0\leq r$. Using Proposition \ref{propcentral}, we directly deduce the following corollary:

\begin{corollary}\label{coro1}
Let's consider $1< l_0 \leq r-1$, $\tilde \omega_k^r(j)$ defined in Eq. \eqref{pesosr1finales} and $p^{r}_k$ with $k=0,\hdots,r-1$, the polynomials which approximate $f$ in the cell-averages at the stencils $S_{k}^r$ ($r$ cells).
 If $f$ is smooth in $[x^{l-1}_{j-r},x^{l-1}_{j+r-1}]\setminus \Omega$ and $f$ has a discontinuity at $\Omega$ then,
\begin{equation}
\left(\mathcal{D}_l\left(\sum_{k=0}^{r-1}\tilde \omega^r_k(j) p^r_k\right)\right)_{2j-1}-\bar f^l_{2j-1}=\left\{
                                                  \begin{array}{ll}
                                                    O(h_{l-1}^{2r-1}), & \hbox{if $\,\,\Omega=\emptyset$;} \\
O(h_{l-1}^{r+l_0-1}), & \hbox{if   $\,\,\Omega=[x^{l-1}_{j+l_0-1},x^{l-1}_{j+l_0}]$. }\\
                                                  \end{array}
                                                \right.
\end{equation}
\end{corollary}

Thus, we have constructed a new prediction operator in the cell-average setting using the primitive function with progressive order of accuracy at zones affected by discontinuities. In the next section, we apply this algorithm to the compression and interpolation of univariate and bivariate functions.

\section{Numerical experiments}\label{num_exp}
This section is dedicated to check the performance of the new WENO-($2r-1$)  algorithm versus the classical WENO-($2r-1$)  algorithm in the cell average setting. In the first subsection we will present numerical results related to the accuracy of the new algorithm close to discontinuities. In the second subsection, we will present results for multiresolution of piecewise continuous univariate functions. Finally, in the third subsection we will show some results for multiresolution and bivariate piecewise continuous functions. About the parameters $\epsilon$ and $t$ in (\ref{pesos}), we have set $\epsilon=10^{-16}$ and $t=r$ for a stencil of $2r-1$ cells.


\subsection{Numerical analysis of the order of accuracy close to discontinuities}
When trying to check the accuracy of an adaptive WENO-($2r-1$)  algorithm that works with piecewise smooth functions in the cell averages, we need to consider two cases: when a discontinuity falls at the end of one of the cells of the stencil and when it falls at some point in the middle of one cell. The accuracy that can be attained at the cell that contains the discontinuity is different in both cases, being $O(h^r)$ for the first case, as there is always one smooth sub-stencil of $r$ cells, and $O(1)$ for the second case, as there is no smooth sub-stencil available.

\begin{figure}[!ht]
\centerline{\psfig{figure=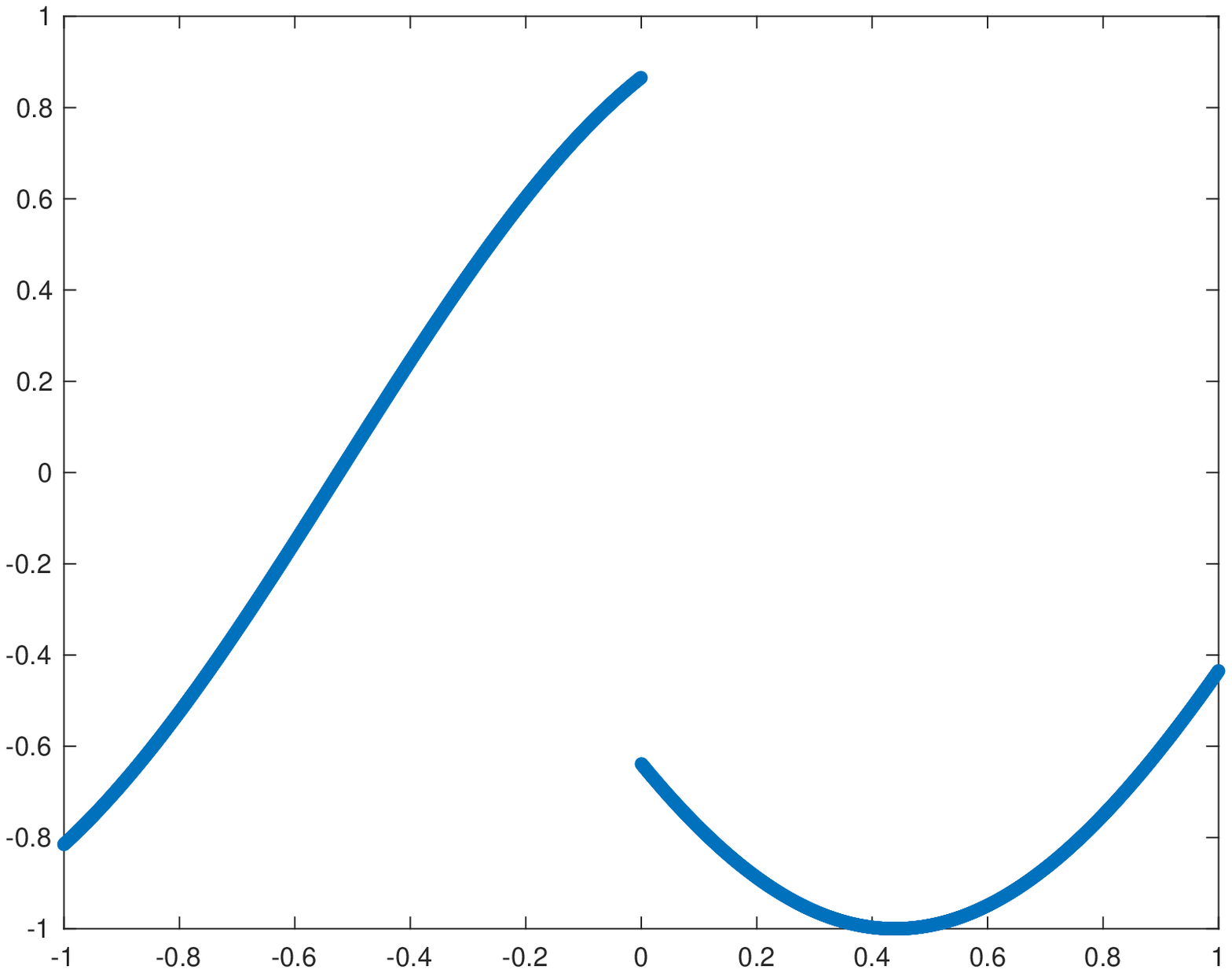,height=4.5cm}\\
\psfig{figure=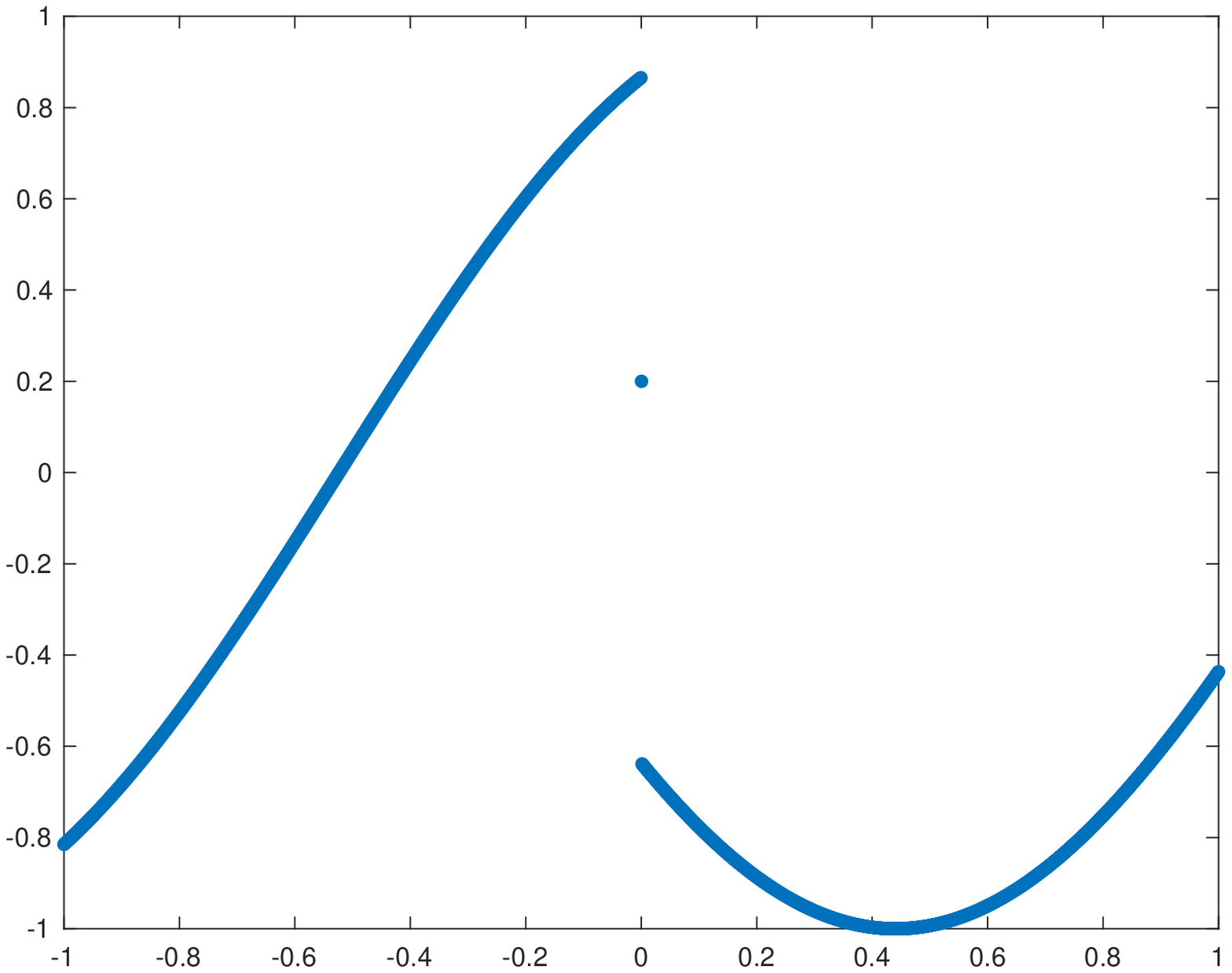,height=4.5cm}
}
\caption{In this figure we show the plot of the functions in (\ref{experimento1}) and (\ref{experimento2}) that will be used for the experiments presented in this section.}\label{funciones}
\end{figure}

{\it\bf Example 1} In this example we will consider the data presented in Figure (\ref{funciones}) to the left.
\begin{equation}\label{experimento1}
f(x)=\left\{\begin{array}{ll}
\sin(2x+\pi/3), & \textrm{ if } -1\leq x<0,\\
\sin(2x-\pi/7), & \textrm{ if } 0\leq x\leq1.
\end{array}\right.
\end{equation}
We will consider that the sampling in the point values of the function in (\ref{experimento1}) corresponds to the cell averages $\bar f^l_j$, obtained through (\ref{discretiza}), of a function $f(x)$ that is unknown. We can see that the function in (\ref{experimento1}) presents a jump. In the cell averages this fact corresponds to the case when the discontinuity falls at the end or the beginning of one cell. We can easily check the numerical accuracy attained by the algorithms that we are studying close to the discontinuities using a grid refinement analysis. We will compare the order of accuracy attained by the new WENO-($2r-1$)  algorithm for $r=3, 4, 5$ with the one obtained by the classical WENO-($2r-1$)  algorithm. The grid refinement analysis can be done following the next steps:
\begin{itemize}
\item We interpolate the initial data vector $v^i$ of $2^i$ cells and we obtain the vector $\hat{v}^{i+1} $ with $2^{i+1}$ cells.
\item The error at each spatial cell $j$ and at the resolution scale $i+1$ can be written as $e_j^{i+1}=|v^{i+1}_j-\hat{v}^{i+1}_j|$.
\item The order of accuracy at the resolution scale $i+1$ and at the cell $j$ is denoted as $o^{i+1}_j$ and computed using the following formula, $$o^{i+1}_j=\log_2\left(\frac{e^{i+1}_{2j}}{e^i_j}\right).$$
\end{itemize}
Tables \ref{tabla_ex1_1}, \ref{tabla_ex1_3} and \ref{tabla_ex1_5} show the results obtained by the classical WENO-($2r-1$)  algorithm for $r=3, 4, 5$ and for the data obtained from (\ref{experimento1}). In these tables we can see that at the places where the errors and orders have been written in bold there is enough information to obtain an improved accuracy. The reason is that classical WENO-($2r-1$)  algorithm is not designed to optimize the accuracy when the stencil is affected by a discontinuity. Tables \ref{tabla_ex1_2}, \ref{tabla_ex1_4} and \ref{tabla_ex1_6} show the results obtained by the new algorithm. In this case we can also observe a reduction of the accuracy close to the discontinuity. Even so, we can see that reduction of accuracy is adjusted to the number of smooth cells available, so the accuracy close to the discontinuity has been clearly improved if we compare it with the results obtained by the classical WENO-($2r-1$)  algorithm. It is also important to mention that, in this particular case (due to the position of the discontinuity at the end of one cell), the minimum accuracy reached by both algorithms close to the discontinuity is $O(h^{r})$. The last cells to the right of Tables \ref{tabla_ex1_1}, \ref{tabla_ex1_2}, \ref{tabla_ex1_3}, \ref{tabla_ex1_4} and \ref{tabla_ex1_5}  are dedicated to present the computational time (in seconds) obtained by each algorithm. In order to obtain the computational time, we have executed the algorithm 1000 times and computed the mean of all the results. We can see that the computational cost is similar for both algorithms.

\begin{table}[!ht]
\begin{center}
\resizebox{15cm}{!} {
\begin{tabular}{|c|c|c|c|c|c|c|c|c|c|c|c|c|c|c|c|c|c|c|c|c|c|c|c|c|c|c|c|}
\hline\multicolumn{1}{|c|}{ }& \multicolumn{2}{|c|}{$x_{2j-6}$} & \multicolumn{2}{|c|}{$x_{2j-4}$}  & \multicolumn{2}{|c|}{$x_{2j-2}$}& \multicolumn{2}{|c|}{$x_{2j}$}& \multicolumn{2}{|c|}{$x_{2j+2}$}& \multicolumn{2}{|c|}{$x_{2j+4}$} &\multirow{ 2}{*}{Comp. t. (s)}
              \\
\cline{1-13} $i$ & $e^i_{2j-6}$ & $o^i_{2j-6}$ & $e^i_{2j-4}$ & $o^i_{2j-4}$ & $e^i_{2j-2}$ & $o^i_{2j-2}$ & $e^i_{2j}$ & $o^i_{2j}$ & $e^i_{2j+2}$ & $o^i_{2j+2}$ & $e^i_{2j+4}$ & $o^i_{2j+4}$ &
              \\
\hline 4& 8.367e-05 & - & 2.254e-04 & - & 9.116e-04 & - & 6.402e-04 & - & 2.148e-04 & - & 1.065e-05 & - &2.445e-04
            \\
\hline 5&2.131e-06 & 5.295 & 2.148e-05 & 3.391 &  9.667e-05 & 3.237 &  1.002e-04  & 2.675 &  2.670e-05  & 3.008& 1.419e-06 & 2.908 &2.287e-04
            \\
\hline 6&5.762e-08 & 5.209 & 2.407e-06 & 3.158 &  1.086e-05 & 3.154 &  1.367e-05  & 2.874 &  3.323e-06  & 3.006& 6.079e-08 & 4.544 &3.330e-04
            \\
\hline 7&1.641e-09 & 5.134 & 2.871e-07 & 3.068 &  1.277e-06 & 3.088 &  1.775e-06  & 2.945 &  4.183e-07  & 2.990& 2.089e-09 & 4.863 &5.655e-04
            \\
\hline 8&4.854e-11 & 5.079 & 3.512e-08 & 3.031 &  1.544e-07 & 3.047 &  2.259e-07  & 2.974 &  5.261e-08  & 2.991& 6.784e-11 & 4.944 & 8.927e-04
            \\
\hline 9&1.472e-12 & 5.044 & 4.347e-09 & 3.015 &  1.898e-08 & 3.024 &  2.849e-08  & 2.987 &  6.602e-09  & 2.994& 2.158e-12 & 4.975 & 1.475e-03
            \\
\hline 10&{ 4.530e-14} & { 5.022} & {\bf 5.407e-10} & {\bf 3.007} &  2.352e-09 & 3.012 &  3.576e-09  & 2.994 &  {\bf 8.269e-10}  & {\bf 2.997}& { 6.783e-14} & {4.991} & 2.652e-03
	    \\
\hline
\end{tabular}
}
\caption{In this table we present a grid refinement analysis for the classical implementation of WENO-6 algorithm. The data has been obtained from the sampling of the function in (\ref{experimento1}), interpreting that this data comes from a cell average discretization of a piecewise continuous and unknown function. In this case, the discontinuity falls at the end of one of the cells.}\label{tabla_ex1_1}
\end{center}
\end{table}

\begin{table}[!ht]
\begin{center}
\resizebox{15cm}{!} {
\begin{tabular}{|c|c|c|c|c|c|c|c|c|c|c|c|c|c|c|c|c|c|c|c|c|c|c|c|c|c|c|c|}
\hline\multicolumn{1}{|c|}{ }& \multicolumn{2}{|c|}{$x_{2j-6}$} & \multicolumn{2}{|c|}{$x_{2j-4}$}  & \multicolumn{2}{|c|}{$x_{2j-2}$}& \multicolumn{2}{|c|}{$x_{2j}$}& \multicolumn{2}{|c|}{$x_{2j+2}$}& \multicolumn{2}{|c|}{$x_{2j+4}$} &\multirow{ 2}{*}{Comp. t. (s)}
              \\
\cline{1-13} $i$ & $e^i_{2j-6}$ & $o^i_{2j-6}$ & $e^i_{2j-4}$ & $o^i_{2j-4}$ & $e^i_{2j-2}$ & $o^i_{2j-2}$ & $e^i_{2j}$ & $o^i_{2j}$ & $e^i_{2j+2}$ & $o^i_{2j+2}$ & $e^i_{2j+4}$ & $o^i_{2j+4}$ &
              \\
\hline 4&  8.228e-05 & - & 6.612e-07 & - & 9.117e-04 & - & 6.396e-04 & - & 1.215e-04 & - & 2.508e-05 & - &3.422e-04
            \\
\hline 5&2.097e-06 & 5.294 & 2.760e-06 & -2.062 &  9.667e-05 & 3.237 &  1.002e-04  & 2.674 &  6.475e-06  & 4.231& 1.345e-06 & 4.221 &2.213e-04
            \\
\hline 6&5.719e-08 & 5.196 & 2.439e-07 & 3.501 &  1.086e-05 & 3.154 &  1.367e-05  & 2.874 &  3.228e-07  & 4.326& 6.041e-08 & 4.476 &3.491e-04
            \\
\hline 7&1.637e-09 & 5.127 & 1.736e-08 & 3.812 &  1.277e-06 & 3.088 &  1.775e-06  & 2.945 &  1.729e-08  & 4.222& 2.086e-09 & 4.856 & 6.525e-04
            \\
\hline 8&4.851e-11 & 5.077 & 1.149e-09 & 3.917 &  1.544e-07 & 3.047 &  2.259e-07  & 2.974 &  9.877e-10  & 4.130& 6.783e-11 & 4.943 &1.020e-03
            \\
\hline 9&1.471e-12 & 5.043 & 7.381e-11 & 3.961 &  1.898e-08 & 3.024 &  2.849e-08  & 2.987 &  5.878e-11  & 4.071& 2.158e-12 & 4.974 & 1.721e-03
            \\
\hline 10&{4.530e-14} & {5.022} & {\bf4.674e-12} & {\bf 3.981} &  2.352e-09 & 3.012 &  3.576e-09  & 2.994 &  {\bf3.581e-12 } & {\bf 4.037}& { 6.783e-14} & { 4.991} &3.331e-03
	    \\
\hline
\end{tabular}
}
\caption{In this table we present a grid refinement analysis for the new implementation of WENO-6 algorithm. The data has been obtained from the sampling of the function in (\ref{experimento1}), interpreting that this data comes from a cell average discretization of a piecewise continuous and unknown function. In this case, the discontinuity falls at the end of one of the cells.}\label{tabla_ex1_2}
\end{center}
\end{table}


\begin{sidewaystable}
  \centering
\resizebox{16cm}{!} {
\begin{tabular}{|c|c|c|c|c|c|c|c|c|c|c|c|c|c|c|c|c|c|c|c|c|c|c|c|c|c|c|c}
\hline\multicolumn{1}{|c|}{ }&\multicolumn{2}{|c|}{$x_{2j-8}$}& \multicolumn{2}{|c|}{$x_{2j-6}$} & \multicolumn{2}{|c|}{$x_{2j-4}$}  & \multicolumn{2}{|c|}{$x_{2j-2}$}& \multicolumn{2}{|c|}{$x_{2j}$}& \multicolumn{2}{|c|}{$x_{2j+2}$}& \multicolumn{2}{|c|}{$x_{2j+4}$}&  \multicolumn{2}{|c|}{$x_{2j+6}$} &\multirow{ 2}{*}{Comp. t. (s)}
              \\
\cline{1-17} $i$ &$e^i_{2j-8}$ & $o^i_{2j-8}$ & $e^i_{2j-6}$ & $o^i_{2j-6}$ & $e^i_{2j-4}$ & $o^i_{2j-4}$ & $e^i_{2j-2}$ & $o^i_{2j-2}$ & $e^i_{2j}$ & $o^i_{2j}$ & $e^i_{2j+2}$ & $o^i_{2j+2}$ & $e^i_{2j+4}$ & $o^i_{2j+4}$  & $e^i_{2j+6}$ & $o^i_{2j+6}$ &
              \\

\hline 4& 2.554e-07 & - & 9.822e-06 & - & 3.348e-05 & - & 1.215e-04 & - & 1.918e-04 & - & 5.219e-05 & - & 5.993e-06 & - &  9.656e-08 & - &2.994e-04
            \\
\hline 5& 6.463e-09 & 5.304 &6.978e-07 & 3.815 & 2.535e-06 & 3.723 &  9.835e-06 & 3.626 &  1.054e-05  & 4.185 &  2.600e-06  & 4.327& 6.484e-07 & 3.208 & 2.510e-08 & 1.944 &2.699e-04
            \\
\hline 6& 8.849e-11 & 6.191 &4.599e-08 & 3.923 & 1.705e-07 & 3.895 &  6.737e-07 & 3.868 &  6.002e-07  & 4.135 &  1.480e-07  & 4.135& 3.870e-08 & 4.067 & 1.099e-10 & 7.835&3.972e-04
            \\
\hline 7& 9.164e-13 & 6.593 &2.947e-09 & 3.964 & 1.100e-08 & 3.954 &  4.375e-08 & 3.945 &  3.548e-08  & 4.080 &  8.798e-09  & 4.072& 2.325e-09 & 4.057 & 6.374e-13 & 7.430 &6.571e-04
            \\
\hline 8& { 8.660e-15} & { 6.725} &{\bf1.864e-10} &{\bf 3.983} & 6.974e-10 & 3.979 &  2.783e-09 & 3.975 &  2.151e-09  & 4.044 &  5.355e-10  & 4.038&{\bf 1.421e-10} & {\bf4.032}& { 5.107e-15} & { 6.964} &1.139e-03
            \\
\hline 9& 1.776e-15 & 2.285 &1.172e-11 & 3.992 & {\bf 4.389e-11} & {\bf 3.990} &  1.754e-10 & 3.988 &  1.324e-10  & 4.023 &  3.302e-11  & 4.020& 8.786e-12 & 4.016 & 1.887e-15 & 1.436 & 1.965e-03
            \\
\hline 10& 1.887e-15 & -0.087 &7.431e-13 &  3.979 & 2.748e-12 & 3.997 &  1.100e-11 & 3.994 &  8.208e-12  & 4.011 & {\bf 2.038e-12}  & {\bf 4.018}& 5.470e-13 & { 4.005} & 1.443e-15 & 0.387 &3.457e-03
	    \\
\hline
\end{tabular}
}
\captionof{table}{In this table we present a grid refinement analysis for the classical implementation of WENO-8 algorithm. The data has been obtained from the sampling of the function in (\ref{experimento1}), interpreting that this data comes from a cell average discretization of a piecewise continuous and unknown function. In this case, the discontinuity falls at the end of one of the cells.}\label{tabla_ex1_3}
\vspace{0.5cm}

\resizebox{16cm}{!} {
\begin{tabular}{|c|c|c|c|c|c|c|c|c|c|c|c|c|c|c|c|c|c|c|c|c|c|c|c|c|c|c|c|}
\hline\multicolumn{1}{|c|}{ }&\multicolumn{2}{|c|}{$x_{2j-8}$}& \multicolumn{2}{|c|}{$x_{2j-6}$} & \multicolumn{2}{|c|}{$x_{2j-4}$}  & \multicolumn{2}{|c|}{$x_{2j-2}$}& \multicolumn{2}{|c|}{$x_{2j}$}& \multicolumn{2}{|c|}{$x_{2j+2}$}& \multicolumn{2}{|c|}{$x_{2j+4}$}&  \multicolumn{2}{|c|}{$x_{2j+6}$} &\multirow{ 2}{*}{Comp. t. (s)}
              \\
\cline{1-17} $i$ &$e^i_{2j-8}$ & $o^i_{2j-8}$ & $e^i_{2j-6}$ & $o^i_{2j-6}$ & $e^i_{2j-4}$ & $o^i_{2j-4}$ & $e^i_{2j-2}$ & $o^i_{2j-2}$ & $e^i_{2j}$ & $o^i_{2j}$ & $e^i_{2j+2}$ & $o^i_{2j+2}$ & $e^i_{2j+4}$ & $o^i_{2j+4}$  & $e^i_{2j+6}$ & $o^i_{2j+6}$ &
              \\
\hline 4&  3.031e-07 & - & 1.312e-07 & - & 1.183e-05 & - & 1.215e-04 & - & 1.918e-04 & - & 5.764e-06 & - & 9.661e-06 & - &  8.490e-08 & -  &2.865e-04
            \\
\hline 5& 6.649e-09 & 5.510 &6.568e-09 & 4.320 & 2.980e-07 & 5.310 &  9.835e-06 & 3.626 &  1.054e-05  & 4.185 &  1.972e-07  & 4.869& 4.340e-08 & 7.798 & 2.542e-08 & 1.740&2.966e-04
            \\
\hline 6& 8.904e-11 & 6.223 &1.459e-10 & 5.493 & 8.003e-09 & 5.219 &  6.737e-07 & 3.868 &  6.002e-07  & 4.135 &  8.572e-09  & 4.524& 3.336e-10 & 7.024 & 1.105e-10 & 7.845 & 4.343e-04
            \\
\hline 7& 9.177e-13 & 6.600 &2.889e-12 & 5.658 & 2.283e-10 & 5.131 &  4.375e-08 & 3.945 &  3.548e-08  & 4.080 &  2.935e-10  & 4.868& 3.825e-12 & 6.446 & 6.385e-13  & 7.436 &7.337e-04
            \\
\hline 8& { 8.660e-15} &{6.728} &{\bf 5.196e-14} & {\bf 5.797} & 6.773e-12 & 5.075 &  2.783e-09 & 3.975 &  2.151e-09  & 4.044 &  9.515e-12  & 4.947& {\bf 5.063e-14} & {\bf 6.239} & {4.219e-15} & {7.242} &1.305e-03
            \\
\hline 9& 1.776e-15 & 2.285 &4.441e-16 & 6.870 & {\bf 2.052e-13} & {\bf 5.045} &  1.754e-10 & 3.988 &  1.324e-10  & 4.023 &  3.046e-13  & 4.965& 3.220e-15 & 3.975 & 1.110e-16 & 5.248 &2.221e-03
            \\
\hline 10& 5.218e-15 & -1.555 &4.108e-15 & -3.209 & 1.665e-15 & 6.945 &  1.100e-11 & 3.994 &  8.208e-12  & 4.011 &  {\bf 8.105e-15}  & {\bf 5.232}& 6.994e-15 & -1.119 & 5.662e-15 & -5.672&4.078e-03
	    \\
\hline
\end{tabular}
}
\caption{In this table we present a grid refinement analysis for the new implementation of WENO-8 algorithm. The data has been obtained from the sampling of the function in (\ref{experimento1}), interpreting that this data comes from a cell average discretization of a piecewise continuous and unknown function. In this case, the discontinuity falls at the end of one of the cells.}\label{tabla_ex1_4}

\vspace{0.5cm}

\resizebox{17cm}{!} {
\begin{tabular}{|c|c|c|c|c|c|c|c|c|c|c|c|c|c|c|c|c|c|c|c|c|c|c|c|c|c|c|c}
\hline\multicolumn{1}{|c|}{ }&\multicolumn{2}{|c|}{$x_{2j-10}$} &\multicolumn{2}{|c|}{$x_{2j-8}$}& \multicolumn{2}{|c|}{$x_{2j-6}$} & \multicolumn{2}{|c|}{$x_{2j-4}$}  & \multicolumn{2}{|c|}{$x_{2j-2}$}& \multicolumn{2}{|c|}{$x_{2j}$}& \multicolumn{2}{|c|}{$x_{2j+2}$}& \multicolumn{2}{|c|}{$x_{2j+4}$}&  \multicolumn{2}{|c|}{$x_{2j+6}$} & \multicolumn{2}{|c|}{$x_{2j+8}$} &\multirow{ 2}{*}{Comp. t. (s)}  
              \\
\cline{1-21} $i$ &$e^i_{2j-10}$ & $o^i_{2j-10}$ &$e^i_{2j-8}$ & $o^i_{2j-8}$ & $e^i_{2j-6}$ & $o^i_{2j-6}$ & $e^i_{2j-4}$ & $o^i_{2j-4}$ & $e^i_{2j-2}$ & $o^i_{2j-2}$ & $e^i_{2j}$ & $o^i_{2j}$ & $e^i_{2j+2}$ & $o^i_{2j+2}$ & $e^i_{2j+4}$ & $o^i_{2j+4}$  & $e^i_{2j+6}$ & $o^i_{2j+6}$  & $e^i_{2j+8}$ & $o^i_{2j+8}$& 
              \\
\hline 4& 1.758e-03 & - & 6.248e-07 & - & 2.681e-06 & - & 8.016e-06 & - & 3.495e-05 & - & 1.268e-05 & - & 3.259e-06 & - & 1.227e-06 & - &  3.166e-07 & - &  1.421e-03-01  & -&3.014e-04
            \\
\hline 5&5.335e-11 & 24.974 & 1.689e-08 & 5.209 &6.826e-08 & 5.296 & 2.042e-07 & 5.295 &  8.979e-07 & 5.283 &  7.086e-07  & 4.162 &  1.682e-07  & 4.277& 5.872e-08 & 4.385 & 1.518e-08 & 4.383 &3.916e-12 & 28.435  &2.649e-04
            \\
\hline 6&{8.660e-14} & {9.267} & 4.631e-10 & 5.189 &1.849e-09 & 5.206 & 5.499e-09 & 5.215 &  2.409e-08 & 5.220 &  2.644e-08  & 4.744 &  6.164e-09  & 4.770& 2.118e-09 & 4.793 & 5.418e-10 & 4.808 &6.606e-14 & 5.890& 4.217e-04
            \\
\hline 7&7.772e-16 & 6.800 & {\bf 1.336e-11} & {\bf 5.115} &{\bf 5.305e-11} & {\bf 5.124} & 1.569e-10 & 5.131 &  6.845e-10 & 5.137 &  8.871e-10  & 4.897 &  2.056e-10  & 4.906& {\bf 7.027e-11} & {\bf 4.914} & {\bf 1.790e-11} & {\bf 4.920} &{1.110e-16} &{9.217}& 7.210e-04
            \\
\hline 8&2.220e-16 & 1.807 & 3.990e-13 & 5.066 &1.581e-12 & 5.068 & {\bf 4.665e-12} & {\bf 5.072} &  2.028e-11 & 5.077 &  2.862e-11  & 4.954 &  {\bf 6.620e-12}  & {\bf 4.957}& 2.257e-12 & 4.961 & 5.731e-13 & 4.965 &3.331e-16 & -1.585&1.250e-03
            \\
\hline 9&1.332e-15 & -2.585 & 1.421e-14 & 4.811 &4.574e-14 & 5.112 & 1.412e-13 & 5.046 &  6.202e-13 & 5.031 &  9.130e-13  & 4.970 &  2.123e-13  & 4.963& 6.783e-14 & 5.056 & 1.765e-14 & 5.021 &1.998e-15 & -2.585 &2.176e-03
            \\
\hline 10&9.992e-15 & -2.907 & 1.887e-15 & 2.913 &4.108e-15 & 3.477 & 1.665e-15 & 6.406 &  1.799e-14 & 5.108 &  2.043e-14  & 5.482 &  8.105e-15  & 4.711& 6.994e-15 & 3.278 & 1.443e-15 & 3.612 &4.441e-16 & 2.170 &3.856e-03
	    \\
\hline
\end{tabular}
}
\caption{In this table we present a grid refinement analysis for the classical implementation of WENO-10 algorithm. The data has been obtained from the sampling of the function in (\ref{experimento1}), interpreting that this data comes from a cell average discretization of a piecewise continuous and unknown function. In this case, the discontinuity falls at the end of one of the cells.}\label{tabla_ex1_5}
\vspace{0.5cm}

\resizebox{17cm}{!} {
\begin{tabular}{|c|c|c|c|c|c|c|c|c|c|c|c|c|c|c|c|c|c|c|c|c|c|c|c|c|c|c|c|}
\hline\multicolumn{1}{|c|}{ }&\multicolumn{2}{|c|}{$x_{2j-10}$} &\multicolumn{2}{|c|}{$x_{2j-8}$}& \multicolumn{2}{|c|}{$x_{2j-6}$} & \multicolumn{2}{|c|}{$x_{2j-4}$}  & \multicolumn{2}{|c|}{$x_{2j-2}$}& \multicolumn{2}{|c|}{$x_{2j}$}& \multicolumn{2}{|c|}{$x_{2j+2}$}& \multicolumn{2}{|c|}{$x_{2j+4}$}&  \multicolumn{2}{|c|}{$x_{2j+6}$} & \multicolumn{2}{|c|}{$x_{2j+8}$} &\multirow{ 2}{*}{Comp. t. (s)}   
              \\
\cline{1-21} $i$ &$e^i_{2j-10}$ & $o^i_{2j-10}$ &$e^i_{2j-8}$ & $o^i_{2j-8}$ & $e^i_{2j-6}$ & $o^i_{2j-6}$ & $e^i_{2j-4}$ & $o^i_{2j-4}$ & $e^i_{2j-2}$ & $o^i_{2j-2}$ & $e^i_{2j}$ & $o^i_{2j}$ & $e^i_{2j+2}$ & $o^i_{2j+2}$ & $e^i_{2j+4}$ & $o^i_{2j+4}$  & $e^i_{2j+6}$ & $o^i_{2j+6}$  & $e^i_{2j+8}$ & $o^i_{2j+8}$& 
              \\
\hline 4& 1.758e-03 & - & 2.471e-08 & - & 1.451e-07 & - & 6.938e-07 & - & 3.495e-05 & - & 1.268e-05 & - & 2.086e-06 & - & 1.230e-08 & - &  5.829e-08 & - &  1.421e-03  & - &3.270e-04
            \\
\hline 5&5.266e-11 & 24.993 & 9.451e-12 & 11.352 &1.015e-09 & 7.160 & 2.153e-08 & 5.010 &  8.979e-07 & 5.283 &  7.086e-07  & 4.162 &  2.886e-08  & 6.175& 5.400e-10 & 4.509 & 7.647e-11 & 9.574 &1.265e-12 & 30.066 &3.214e-04
            \\
\hline 6&{8.660e-14} & {9.248} & 1.049e-13 & 6.493 &6.783e-12 & 7.225 & 4.019e-10 & 5.744 &  2.409e-08 & 5.220 &  2.644e-08  & 4.744 &  4.002e-10  & 6.172& 6.233e-12 & 6.437 & 2.708e-13 & 8.142 &6.428e-14 & 4.298& 5.006e-04
            \\
\hline 7&3.331e-16 & 8.022 & {\bf 6.661e-16} & {\bf 7.299} &{\bf 4.707e-14} & {\bf 7.171} & 6.721e-12 & 5.902 &  6.845e-10 & 5.137 &  8.871e-10  & 4.897 &  5.771e-12  & 6.116& {\bf 5.596e-14} & {\bf 6.800} & {\bf 7.772e-16} & {\bf 8.445} &{1.110e-16} &{9.177}&8.985e-04
            \\
\hline 8&2.220e-16 & 0.585 & 2.220e-16 & 1.585 &4.441e-16 & 6.728 & {\bf 1.082e-13} & {\bf 5.956} &  2.028e-11 & 5.077 &  2.862e-11  & 4.954 &  {\bf 8.626e-14}  & {\bf 6.064}& 8.882e-16 & 5.972 & 6.661e-16 & 0.415 &3.331e-16 & 1.222& 1.606e-03
            \\
\hline 9&2.220e-15 & -3.322 & 3.553e-15 & -4.000 &4.441e-16 & 0.000 & 7.994e-15 & 3.759 &  6.202e-13 & 5.031 &  9.130e-13  & 4.970 &  2.665e-15  & 5.017& 3.331e-16 & 1.415 & 1.887e-15 & -3.087 &3.775e-15 & -2.766 & 2.898e-03
            \\
\hline 10&2.887e-15 & -0.379 & 1.887e-15 & 0.913 &2.998e-15 & -2.755 & 1.665e-15 & 2.263 &  1.799e-14 & 5.108 &  2.043e-14  & 5.482 &  9.992e-16  & 1.415& 6.994e-15 & -4.392 & 8.549e-15 & -2.179 &7.550e-15 & -1.000&5.314e-03
	    \\
\hline
\end{tabular}
}
\caption{In this table we present a grid refinement analysis for the new implementation of WENO-10 algorithm. The data has been obtained from the sampling of the function in (\ref{experimento1}), interpreting that this data comes from a cell average discretization of a piecewise continuous and unknown function. In this case, the discontinuity falls at the end of one of the cells.}\label{tabla_ex1_6}
\end{sidewaystable}
\pagebreak

{\it\bf Example 2} Let's continue with the function plotted in Figure (\ref{funciones}) to the right,
\begin{equation}\label{experimento2}
f(x)=\left\{\begin{array}{ll}
\sin(2x+\pi/3), & \textrm{ if } -1\leq x<0,\\
0.2, & \textrm{ if } x=0,\\
\sin(2x-\pi/7), & \textrm{ if } 0< x\leq1.
\end{array}\right.
\end{equation}
In this case, the discontinuity is placed at some point in the middle of one cell and the minimum accuracy attained by both algorithms is $O(1)$. Tables \ref{tabla_ex2_1}, \ref{tabla_ex2_2}, \ref{tabla_ex2_3}, \ref{tabla_ex2_4}, \ref{tabla_ex2_5} and \ref{tabla_ex2_6} show the results of this experiment. Apart from the fact that the spatial pattern of orders of accuracy changes due to the position of the discontinuity, the conclusions are the same as in the previous experiment. The classical WENO-$(2r-1)$ algorithm does not optimize the order of accuracy while the new WENO-$(2r-1)$ algorithm optimizes the order of accuracy and at the same time provides an essentially non oscillatory result. The computational cost is similar for both algorithms.

\begin{table}[!ht]
\begin{center}
\resizebox{15cm}{!} {
\begin{tabular}{|c|c|c|c|c|c|c|c|c|c|c|c|c|c|c|c|c|c|c|c|c|c|c|c|c|c|c|c|c|c|}
\hline\multicolumn{1}{|c|}{ }& \multicolumn{2}{|c|}{$x_{2j-6}$} & \multicolumn{2}{|c|}{$x_{2j-4}$}  & \multicolumn{2}{|c|}{$x_{2j-2}$}& \multicolumn{2}{|c|}{$x_{2j}$}& \multicolumn{2}{|c|}{$x_{2j+2}$}& \multicolumn{2}{|c|}{$x_{2j+4}$} & \multicolumn{2}{|c|}{$x_{2j+6}$} &\multirow{ 2}{*}{Comp. t. (s)}
              \\
\cline{1-15} $i$ & $e^i_{2j-6}$ & $o^i_{2j-6}$ & $e^i_{2j-4}$ & $o^i_{2j-4}$ & $e^i_{2j-2}$ & $o^i_{2j-2}$ & $e^i_{2j}$ & $o^i_{2j}$ & $e^i_{2j+2}$ & $o^i_{2j+2}$ & $e^i_{2j+4}$ & $o^i_{2j+4}$  & $e^i_{2j+6}$ & $o^i_{2j+6}$ &
              \\
\hline 4&  8.367e-05 & - & 2.254e-04 & - & 9.116e-04 & - & 3.445e-03 & - & 4.408e-04 & - & 1.611e-04 & - &  1.071e-06 & - & 2.269e-04
            \\
\hline 5&2.131e-06 & 5.295 & 2.148e-05 & 3.391 &  9.667e-05 & 3.237 &  8.075e-03  & -1.229 &  1.047e-03  & -1.248& 3.546e-04 & -1.138 & 1.024e-06 & 0.064 & 2.096e-04
            \\
\hline 6&5.762e-08 & 5.209 & 2.407e-06 & 3.158 &  1.086e-05 & 3.154 &  1.634e-02  & -1.017 &  8.523e-04  & 0.297& 6.241e-04 & -0.815 & 5.625e-08 & 4.187& 3.061e-04
            \\
\hline 7&1.641e-09 & 5.134 & 2.871e-07 & 3.068 &  1.277e-06 & 3.088 &  3.270e-02  & -1.001 &  2.225e-05  & 5.259& 7.267e-06 & 6.424 & 2.028e-09 & 4.794 &5.187e-04
            \\
\hline 8&4.854e-11 & 5.079 & 3.512e-08 & 3.031 &  1.544e-07 & 3.047 &  6.540e-02  & -1.000 &  2.522e-07  & 6.463& 5.504e-08 & 7.045 & 6.700e-11 & 4.919 & 8.786e-04
            \\
\hline 9&1.472e-12 & 5.044 & 4.347e-09 & 3.015 &  1.898e-08 & 3.024 &  1.306e-01  & -0.998 &  2.836e-08  & 3.153& 6.570e-09 & 3.067 & 2.147e-12 & 4.964 & 1.532e-03
            \\
\hline 10&4.530e-14 & 5.022 & {\bf 5.407e-10} & {\bf 3.007} &  2.352e-09 & 3.012 &  2.497e-01  & -0.935 &  3.570e-09  & 2.990& {\bf 8.256e-10} & {\bf 2.993} & 6.806e-14 & 4.979 &2.655e-03
	    \\
\hline
\end{tabular}
}
\caption{In this table we present a grid refinement analysis for the classical implementation of WENO-6 algorithm. The data has been obtained from the sampling of the function in (\ref{experimento2}), interpreting that this data comes from a cell average discretization of a piecewise continuous and unknown function. In this case, the discontinuity falls at some point in the middle of the cell that contains de discontinuity.}\label{tabla_ex2_1}
\end{center}
\end{table}

\begin{table}[!ht]
\begin{center}
\resizebox{15cm}{!} {
\begin{tabular}{|c|c|c|c|c|c|c|c|c|c|c|c|c|c|c|c|c|c|c|c|c|c|c|c|c|c|c|c|c|c}
\hline\multicolumn{1}{|c|}{ }& \multicolumn{2}{|c|}{$x_{2j-6}$} & \multicolumn{2}{|c|}{$x_{2j-4}$}  & \multicolumn{2}{|c|}{$x_{2j-2}$}& \multicolumn{2}{|c|}{$x_{2j}$}& \multicolumn{2}{|c|}{$x_{2j+2}$}& \multicolumn{2}{|c|}{$x_{2j+4}$} & \multicolumn{2}{|c|}{$x_{2j+6}$} &\multirow{ 2}{*}{Comp. t. (s)}
              \\
\cline{1-15} $i$ & $e^i_{2j-6}$ & $o^i_{2j-6}$ & $e^i_{2j-4}$ & $o^i_{2j-4}$ & $e^i_{2j-2}$ & $o^i_{2j-2}$ & $e^i_{2j}$ & $o^i_{2j}$ & $e^i_{2j+2}$ & $o^i_{2j+2}$ & $e^i_{2j+4}$ & $o^i_{2j+4}$  & $e^i_{2j+6}$ & $o^i_{2j+6}$ &
              \\
\hline 4& 8.228e-05 & - & 6.612e-07 & - & 9.117e-04 & - & 3.446e-03 & - & 4.265e-04 & - & 1.650e-04 & - &  1.717e-06 & - &3.089e-04
            \\
\hline 5&2.097e-06 & 5.294 & 2.760e-06 & -2.062 &  9.667e-05 & 3.237 &  8.075e-03  & -1.229 &  7.950e-04  & -0.899& 3.935e-04 & -1.254 & 9.221e-07 & 0.897 & 2.632e-04
            \\
\hline 6&5.719e-08 & 5.196 & 2.439e-07 & 3.501 &  1.086e-05 & 3.154 &  1.634e-02  & -1.017 &  4.948e-04  & 0.684& 6.340e-04 & -0.688 & 5.581e-08 & 4.046 &3.370e-04
            \\
\hline 7&1.637e-09 & 5.127 & 1.736e-08 & 3.812 &  1.277e-06 & 3.088 &  3.270e-02  & -1.001 &  1.201e-05  & 5.365& 7.138e-08 & 13.117 & 2.025e-09 & 4.784& 5.917e-04
            \\
\hline 8&4.851e-11 & 5.077 & 1.149e-09 & 3.917 &  1.544e-07 & 3.047 &  6.540e-02  & -1.000 &  2.377e-07  & 5.658& 1.001e-09 & 6.155 & 6.699e-11 & 4.918 &1.002e-03
            \\
\hline 9&1.471e-12 & 5.043 & 7.381e-11 & 3.961 &  1.898e-08 & 3.024 &  1.307e-01  & -0.999 &  2.835e-08  & 3.068& 5.916e-11 & 4.081 & 2.147e-12 & 4.964 &1.732e-03
            \\
\hline 10&4.530e-14 & 5.022 & {\bf 4.674e-12} & {\bf 3.981} &  2.352e-09 & 3.012 &  2.546e-01  & -0.962 &  3.570e-09  & 2.989& {\bf 3.589e-12} & {\bf 4.043} & 6.806e-14 & 4.974 &3.267e-03
	    \\
\hline
\end{tabular}
}
\caption{In this table we present a grid refinement analysis for the new implementation of WENO-6 algorithm. The data has been obtained from the sampling of the function in (\ref{experimento2}), interpreting that this data comes from a cell average discretization of a piecewise continuous and unknown function. In this case, the discontinuity falls at some point in the middle of the cell that contains de discontinuity.}\label{tabla_ex2_2}
\end{center}
\end{table}

\begin{sidewaystable}
  \centering
\resizebox{18cm}{!} {
\begin{tabular}{|c|c|c|c|c|c|c|c|c|c|c|c|c|c|c|c|c|c|c|c|c|c|c|c|c|c|c|c|c|c|}
\hline\multicolumn{1}{|c|}{ }&\multicolumn{2}{|c|}{$x_{2j-8}$}& \multicolumn{2}{|c|}{$x_{2j-6}$} & \multicolumn{2}{|c|}{$x_{2j-4}$}  & \multicolumn{2}{|c|}{$x_{2j-2}$}& \multicolumn{2}{|c|}{$x_{2j}$}& \multicolumn{2}{|c|}{$x_{2j+2}$}& \multicolumn{2}{|c|}{$x_{2j+4}$}&  \multicolumn{2}{|c|}{$x_{2j+6}$} &  \multicolumn{2}{|c|}{$x_{2j+8}$} &\multirow{ 2}{*}{Comp. t. (s)}
              \\
\cline{1-19} $i$ &$e^i_{2j-8}$ & $o^i_{2j-8}$ & $e^i_{2j-6}$ & $o^i_{2j-6}$ & $e^i_{2j-4}$ & $o^i_{2j-4}$ & $e^i_{2j-2}$ & $o^i_{2j-2}$ & $e^i_{2j}$ & $o^i_{2j}$ & $e^i_{2j+2}$ & $o^i_{2j+2}$ & $e^i_{2j+4}$ & $o^i_{2j+4}$  & $e^i_{2j+6}$ & $o^i_{2j+6}$ & $e^i_{2j+8}$ & $o^i_{2j+8}$ &
              \\
\hline 4& 2.554e-07 & - & 9.822e-06 & - & 3.348e-05 & - & 1.215e-04 & - & 3.385e-03 & - & 3.967e-04 & - & 1.391e-04 & - &  3.537e-05 & - &  3.653e-06  & - &2.922e-04
            \\
\hline 5& 6.463e-09 & 5.304 &6.978e-07 & 3.815 & 2.535e-06 & 3.723 &  9.835e-06 & 3.626 &  7.143e-03  & -1.077 &  7.813e-04  & -0.978& 3.027e-04 & -1.122 & 3.962e-05 & -0.164 &3.472e-08 & 6.717 &  2.702e-04
            \\
\hline 6& 8.849e-11 & 6.191 &4.599e-08 & 3.923 & 1.705e-07 & 3.895 &  6.737e-07 & 3.868 &  1.431e-02  & -1.002 &  4.338e-04  & 0.849& 3.422e-04 & -0.177 & 4.314e-06 & 3.199 &1.289e-10 & 8.074 & 3.941e-04
            \\
\hline 7& 9.164e-13 & 6.593 &2.947e-09 & 3.964 & 1.100e-08 & 3.954 &  4.375e-08 & 3.945 &  2.861e-02  & -1.000 &  7.272e-07  & 9.220& 1.426e-06 & 7.906 & 7.068e-09 & 9.254 &6.898e-13 & 7.545 &6.593e-04
            \\
\hline 8& 8.660e-15 & 6.725 &{\bf 1.864e-10} & {\bf 3.983} & 6.974e-10 & 3.979 &  2.783e-09 & 3.975 &  5.723e-02  & -1.000 &  2.233e-09  & 8.347& 6.168e-10 & 11.175 & {\bf 1.448e-10} &{\bf  5.610} &5.773e-15 & 6.901 &1.134e-03
            \\
\hline 9& 1.776e-15 & 2.285 &1.172e-11 & 3.992 & {\bf 4.389e-11} & {\bf 3.990} & {  1.754e-10} & { 3.988} &  1.142e-01  & -0.997 &  1.333e-10  & 4.067& {\bf 3.324e-11 }& {\bf 4.214}& 8.844e-12 & 4.033 &8.882e-16 & 2.700 &1.989e-03
            \\
\hline 10&1.887e-15 & -0.087 &7.431e-13 & 3.979 & { 2.748e-12} & {3.997} &  1.100e-11 & 3.994 &  2.419e-01  & -1.083 &  8.220e-12  & 4.019& 2.043e-12 & 4.024 & 5.383e-13 & 4.038 &8.105e-15 & -3.190 & 3.481e-03
	    \\
\hline
\end{tabular}
}
\caption{In this table we present a grid refinement analysis for the classical implementation of WENO-8 algorithm. The data has been obtained from the sampling of the function in (\ref{experimento2}), interpreting that this data comes from a cell average discretization of a piecewise continuous and unknown function. In this case, the discontinuity falls at some point in the middle of the cell that contains de discontinuity.}\label{tabla_ex2_3}

\vspace{0.5cm}

\resizebox{18cm}{!} {
\begin{tabular}{|c|c|c|c|c|c|c|c|c|c|c|c|c|c|c|c|c|c|c|c|c|c|c|c|c|c|c|c|c|c|}
\hline\multicolumn{1}{|c|}{ }&\multicolumn{2}{|c|}{$x_{2j-8}$}& \multicolumn{2}{|c|}{$x_{2j-6}$} & \multicolumn{2}{|c|}{$x_{2j-4}$}  & \multicolumn{2}{|c|}{$x_{2j-2}$}& \multicolumn{2}{|c|}{$x_{2j}$}& \multicolumn{2}{|c|}{$x_{2j+2}$}& \multicolumn{2}{|c|}{$x_{2j+4}$}&  \multicolumn{2}{|c|}{$x_{2j+6}$} &  \multicolumn{2}{|c|}{$x_{2j+8}$} &\multirow{ 2}{*}{Comp. t. (s)}
              \\
\cline{1-19} $i$ &$e^i_{2j-8}$ & $o^i_{2j-8}$ & $e^i_{2j-6}$ & $o^i_{2j-6}$ & $e^i_{2j-4}$ & $o^i_{2j-4}$ & $e^i_{2j-2}$ & $o^i_{2j-2}$ & $e^i_{2j}$ & $o^i_{2j}$ & $e^i_{2j+2}$ & $o^i_{2j+2}$ & $e^i_{2j+4}$ & $o^i_{2j+4}$  & $e^i_{2j+6}$ & $o^i_{2j+6}$ & $e^i_{2j+8}$ & $o^i_{2j+8}$ &
              \\
\hline 4& 3.031e-07 & - & 1.312e-07 & - & 1.183e-05 & - & 1.215e-04 & - & 3.385e-03 & - & 3.556e-04 & - & 1.142e-04 & - &  6.556e-05 & - &  3.445e-06  & - &2.847e-04
            \\
\hline 5& 6.649e-09 & 5.510 &6.568e-09 & 4.320 & 2.980e-07 & 5.310 &  9.835e-06 & 3.626 &  7.143e-03  & -1.077 &  5.354e-04  & -0.590& 2.033e-04 & -0.831 & 2.130e-05 & 1.622 &3.493e-08 & 6.624 & 3.148e-04
            \\
\hline 6& 8.904e-11 & 6.223 &1.459e-10 & 5.493 & 8.003e-09 & 5.219 &  6.737e-07 & 3.868 &  1.431e-02  & -1.002 &  1.688e-04  & 1.665& 1.162e-04 & 0.807 & 1.251e-08 & 10.734 &1.296e-10 & 8.074 &4.380e-04
            \\
\hline 7& 9.177e-13 & 6.600 &2.889e-12 & 5.658 & 2.283e-10 & 5.131 &  4.375e-08 & 3.945 &  2.861e-02  & -1.000 &  2.668e-07  & 9.305& 4.717e-10 & 17.910 & 3.977e-12 & 11.619 &6.907e-13 & 7.552&  7.414e-04
            \\
\hline 8&  8.660e-15 & 6.728 &{\bf 5.196e-14} & {\bf 5.797 }& 6.773e-12 & 5.075 &  2.783e-09 & 3.975 &  5.723e-02  & -1.000 &  2.201e-09  & 6.922& 9.401e-12 & 5.649 & {\bf 5.262e-14} & {\bf 6.240} &4.885e-15 & 7.143 &1.276e-03
            \\
\hline 9& 1.776e-15 & 2.285 &1.332e-15 & 5.285 & {\bf 2.052e-13} & {\bf 5.045} &  1.754e-10 & 3.988 &  1.144e-01  & -0.999 &  1.333e-10  & 4.046& {\bf 3.013e-13} & {\bf 4.963} & 2.665e-15 & 4.304 &8.882e-16 & 2.459 &2.231e-03
            \\
\hline 10&5.218e-15 & -1.555 &4.108e-15 & -3.209 & 1.665e-15 & 6.945 &  1.100e-11 & 3.994 &  2.414e-01  & -1.078 &  8.220e-12  & 4.019& 1.754e-14 & 4.102 & 5.440e-15 & -1.030 &9.992e-16 & -0.170  &4.160e-03
	    \\
\hline
\end{tabular}
}
\caption{In this table we present a grid refinement analysis for the new implementation of WENO-8 algorithm. The data has been obtained from the sampling of the function in (\ref{experimento2}), interpreting that this data comes from a cell average discretization of a piecewise continuous and unknown function. In this case, the discontinuity falls at some point in the middle of the cell that contains de discontinuity.}\label{tabla_ex2_4}


\vspace{0.5cm}

\resizebox{18cm}{!} {
\begin{tabular}{|c|c|c|c|c|c|c|c|c|c|c|c|c|c|c|c|c|c|c|c|c|c|c|c|c|c|c|c|c|c}
\hline\multicolumn{1}{|c|}{ }&\multicolumn{2}{|c|}{$x_{2j-10}$} &\multicolumn{2}{|c|}{$x_{2j-8}$}& \multicolumn{2}{|c|}{$x_{2j-6}$} & \multicolumn{2}{|c|}{$x_{2j-4}$}  & \multicolumn{2}{|c|}{$x_{2j-2}$}& \multicolumn{2}{|c|}{$x_{2j}$}& \multicolumn{2}{|c|}{$x_{2j+2}$}& \multicolumn{2}{|c|}{$x_{2j+4}$}&  \multicolumn{2}{|c|}{$x_{2j+6}$} & \multicolumn{2}{|c|}{$x_{2j+8}$}&\multirow{ 2}{*}{Comp. t. (s)}    
              \\
\cline{1-21} $i$ &$e^i_{2j-10}$ & $o^i_{2j-10}$ &$e^i_{2j-8}$ & $o^i_{2j-8}$ & $e^i_{2j-6}$ & $o^i_{2j-6}$ & $e^i_{2j-4}$ & $o^i_{2j-4}$ & $e^i_{2j-2}$ & $o^i_{2j-2}$ & $e^i_{2j}$ & $o^i_{2j}$ & $e^i_{2j+2}$ & $o^i_{2j+2}$ & $e^i_{2j+4}$ & $o^i_{2j+4}$  & $e^i_{2j+6}$ & $o^i_{2j+6}$  & $e^i_{2j+8}$ & $o^i_{2j+8}$& 
              \\
\hline 4& 1.750e-03 & - & 6.248e-07 & - & 2.681e-06 & - & 8.016e-06 & - & 3.495e-05 & - & 3.232e-03 & - & 3.247e-04 & - & 1.061e-04 & - &  3.936e-05 & - &  1.349e-03  & -& 2.943e-04
            \\
\hline 5&5.335e-11 & 24.967 & 1.689e-08 & 5.209 &6.826e-08 & 5.296 & 2.042e-07 & 5.295 &  8.979e-07 & 5.283 &  6.439e-03  & -0.994 &  5.766e-04  & -0.829& 1.997e-04 & -0.913 & 4.913e-05 & -0.320 &1.737e-06 & 9.602&   2.621e-04
            \\
\hline 6&{8.660e-14} & {9.267} & 4.631e-10 & 5.189 &1.849e-09 & 5.206 & 5.499e-09 & 5.215 &  2.409e-08 & 5.220 &  1.288e-02  & -1.000 &  7.635e-05  & 2.917& 1.511e-04 & 0.402 & 1.198e-05 & 2.036 &8.119e-08 & 4.419 &4.212e-04
            \\
\hline 7&{ 7.772e-16} & { 6.800} & {\bf 1.336e-11} & {\bf 5.115} &5.305e-11 & 5.124 & 1.569e-10 & 5.131 &  6.845e-10 & 5.137 &  2.575e-02  & -1.000 &  2.528e-09  & 14.882& 7.188e-08 & 11.038 & 5.539e-09 & 11.079 &1.383e-11 & 12.519& 7.151e-04
            \\
\hline 8&2.220e-16 & 1.807 & 3.990e-13 & 5.066 &{\bf 1.581e-12} & {\bf 5.068} & {\bf 4.665e-12} & {\bf 5.072} &  2.028e-11 & 5.077 &  5.151e-02  & -1.000 &  2.826e-11  & 6.483& 6.241e-12 & 13.491 & 2.192e-12 & 11.303 &5.662e-13 & 4.610&  1.247e-03
            \\
\hline 9&1.332e-15 & -2.585 & 1.421e-14 & 4.811 &4.574e-14 & 5.112 & 1.412e-13 & 5.046 &  6.202e-13 & 5.031 &  1.026e-01  & -0.994 &  9.044e-13  & 4.966& {\bf 2.054e-13} &{\bf  4.925} & {\bf 7.017e-14} & {\bf 4.965} &{1.688e-14} & { 5.068}&  2.203e-03
            \\
\hline 10&9.992e-15 & -2.907 & 1.887e-15 & 2.913 &4.108e-15 & 3.477 & 1.665e-15 & 6.406 &  1.799e-14 & 5.108 &  2.384e-01  & -1.217 &  2.942e-14  & 4.942& 1.754e-14 & 3.550 & 1.665e-15 & 5.397 &8.105e-15 & 1.058 &  3.819e-03
	    \\
\hline
\end{tabular}
}
\caption{In this table we present a grid refinement analysis for the classical implementation of WENO-10 algorithm. The data has been obtained from the sampling of the function in (\ref{experimento2}), interpreting that this data comes from a cell average discretization of a piecewise continuous and unknown function. In this case, the discontinuity falls at some point in the middle of the cell that contains de discontinuity.}\label{tabla_ex2_5}

\vspace{0.5cm}

\resizebox{18cm}{!} {
\begin{tabular}{|c|c|c|c|c|c|c|c|c|c|c|c|c|c|c|c|c|c|c|c|c|c|c|c|c|c|c|c|c|c}
\hline\multicolumn{1}{|c|}{ }&\multicolumn{2}{|c|}{$x_{2j-10}$} &\multicolumn{2}{|c|}{$x_{2j-8}$}& \multicolumn{2}{|c|}{$x_{2j-6}$} & \multicolumn{2}{|c|}{$x_{2j-4}$}  & \multicolumn{2}{|c|}{$x_{2j-2}$}& \multicolumn{2}{|c|}{$x_{2j}$}& \multicolumn{2}{|c|}{$x_{2j+2}$}& \multicolumn{2}{|c|}{$x_{2j+4}$}&  \multicolumn{2}{|c|}{$x_{2j+6}$} & \multicolumn{2}{|c|}{$x_{2j+8}$}&\multirow{ 2}{*}{Comp. t. (s)}   
              \\
\cline{1-21} $i$ &$e^i_{2j-10}$ & $o^i_{2j-10}$ &$e^i_{2j-8}$ & $o^i_{2j-8}$ & $e^i_{2j-6}$ & $o^i_{2j-6}$ & $e^i_{2j-4}$ & $o^i_{2j-4}$ & $e^i_{2j-2}$ & $o^i_{2j-2}$ & $e^i_{2j}$ & $o^i_{2j}$ & $e^i_{2j+2}$ & $o^i_{2j+2}$ & $e^i_{2j+4}$ & $o^i_{2j+4}$  & $e^i_{2j+6}$ & $o^i_{2j+6}$  & $e^i_{2j+8}$ & $o^i_{2j+8}$& 
              \\
 \hline 4& 1.750e-03 & - & 2.471e-08 & - & 1.451e-07 & - & 6.938e-07 & - & 3.495e-05 & - & 3.232e-03 & - & 2.486e-04 & - & 6.396e-05 & - &  2.118e-05 & - &  1.349e-03&-  &3.174e-04
            \\
\hline 5&5.266e-11 & 24.986 & 9.450e-12 & 11.353 &1.015e-09 & 7.160 & 2.153e-08 & 5.010 &  8.979e-07 & 5.283 &  6.439e-03  & -0.994 &  3.643e-04  & -0.551& 1.107e-04 & -0.792 & 1.322e-05 & 0.681 &1.331e-08 & 16.630 &  3.505e-04
            \\
\hline 6&{8.660e-14} & {9.248} & 1.049e-13 & 6.493 &6.783e-12 & 7.225 & 4.019e-10 & 5.744 &  2.409e-08 & 5.220 &  1.288e-02  & -1.000 &  1.987e-05  & 4.197& 1.584e-05 & 2.805 & {\bf 7.308e-08} & {\bf 7.498} &2.931e-13 & 15.470 &  5.034e-04
            \\
\hline 7&3.331e-16 & 8.022 & {\bf 6.661e-16} &{\bf  7.299} &4.707e-14 & 7.171 & 6.721e-12 & 5.902 &  6.845e-10 & 5.137 &  2.575e-02  & -1.000 &  1.453e-11  & 20.383& 4.761e-12 & 21.666 & 5.418e-14 & 20.363 &{9.992e-16} & { 8.196} &9.085e-04
            \\
\hline 8&2.220e-16 & 0.585 & 2.220e-16 & 1.585 &{\bf 4.441e-16} &{\bf  6.728} & {\bf 1.082e-13} & {\bf 5.956} &  2.028e-11 & 5.077 &  5.151e-02  & -1.000 &  2.827e-11  & -0.960& 8.815e-14 & 5.755 & 6.661e-16 & 6.346 &4.441e-16 & 1.170 &   1.669e-03
            \\
\hline 9&2.220e-15 & -3.322 & 3.553e-15 & -4.000 &4.441e-16 & 0.000 & 7.994e-15 & 3.759 &  6.202e-13 & 5.031 &  1.029e-01  & -0.998 &  9.044e-13  & 4.966& {\bf 6.661e-16 }& {\bf 7.048 }& 8.882e-16 & -0.415 &2.665e-15 & -2.585 &  2.987e-03
            \\
\hline 10&2.887e-15 & -0.379 & 1.887e-15 & 0.913 &2.998e-15 & -2.755 & 1.665e-15 & 2.263 &  1.799e-14 & 5.108 &  2.384e-01  & -1.212 &  2.942e-14  & 4.942& 3.331e-15 & -2.322 & 5.440e-15 & -2.615 &6.106e-15 & -1.196 &  5.400e-03
	    \\
\hline
\end{tabular}
}
\caption{In this table we present a grid refinement analysis for the new implementation of WENO-10 algorithm. The data has been obtained from the sampling of the function in (\ref{experimento2}), interpreting that this data comes from a cell average discretization of a piecewise continuous and unknown function. In this case, the discontinuity falls at some point in the middle of the cell that contains de discontinuity.}\label{tabla_ex2_6}
\end{sidewaystable}

\subsection{Multiresolution of 1D signals}
Now we will try to analyze the behavior of the new WENO-($2r-1$)  algorithm compared to the classical implementation for multiresolution of univariate piecewise continuous functions. In particular, we will present a compression application and we will analyze the error distribution around discontinuities attained by both algorithms. We are interested in checking if the new WENO-($2r-1$) algorithm manages to reduce the size of the absolute error close to discontinuities. Let us use initial data obtained through the function in (\ref{experimento2}) with an initial resolution of 2048 cells. It is clear that this function presents a cell that contains a discontinuity: if this cell is the central cell of the stencil, all the substencils will be affected by the discontinuity, so it will not be possible to attain an accuracy better than $O(1)$ for any of the algorithms that we are analyzing. In a multiresolution application, if we want to check the error distribution around the discontinuity, we can just keep one detail for each scale and discontinuity, in order to assure that the reconstruction error at the cell that contains the discontinuity will be zero. Table \ref{tabla_mult_2} shows the errors attained in the $L_1$, $L_2$ and $L_\infty$ norms for the classical  and the new implementation of WENO-($2r-1$)  for the particular cases $r=3, 4, 5$. We also show the computational time, in seconds, attained by each of the algorithms. In order to obtain this computational time we have obtained the average of 1000 executions of each of the algorithms. Figure \ref{exp2_mult} presents the error distributions around the discontinuity for each of the experiments presented in table \ref{tabla_mult_2}. We can see that in all the experiments shown, the new implementation of the algorithm obtains better results than the classical implementation. It is particularly interesting to observe in Figure \ref{exp2_mult} how the new algorithm manages to reduce the size of the error and, at the same time, compressing the error distribution towards the discontinuity. The computational time of both algorithms is similar.

\begin{table}[!ht]
\begin{center}
\resizebox{15cm}{!} {
\begin{tabular}{|c|c|c|c|c|c|c|c|c|c|c|c|c|c|c|}
\hline $i$ &WENO-6 & new WENO-6&WENO-8 & new WENO-8&WENO-10 & new WENO-10
              \\
\hline  $l_{\infty}$&2.9636e-05&1.7647e-05& 1.9434e-06&4.7964e-07&9.6089e-08& 2.8984e-09
            \\
\hline  $l_2$&6.037e-12  &1.6761e-12&2.7668e-14&1.3671e-15&5.8535e-17&9.6167e-21
            \\
\hline  $l_1$& 3.2931e-07 &1.5909e-07&2.3539e-08&4.7547e-09&1.0135e-09&7.7522e-12
            \\
\hline  Comp. time (s)& 6.325e-03&7.987e-03 &8.645e-03&9.103e-03&8.986e-03&1.154e-02
            \\
\hline
\end{tabular}
}
\caption{Norms of the error and computational time (in seconds) obtained when compressing the function in (\ref{experimento2}) with four scales of multiresolution and keeping 4 details (one for scale). The initial resolution is 2048 cells.}\label{tabla_mult_2}
\end{center}
\end{table}

\begin{figure}[!ht]
\centerline{\psfig{figure=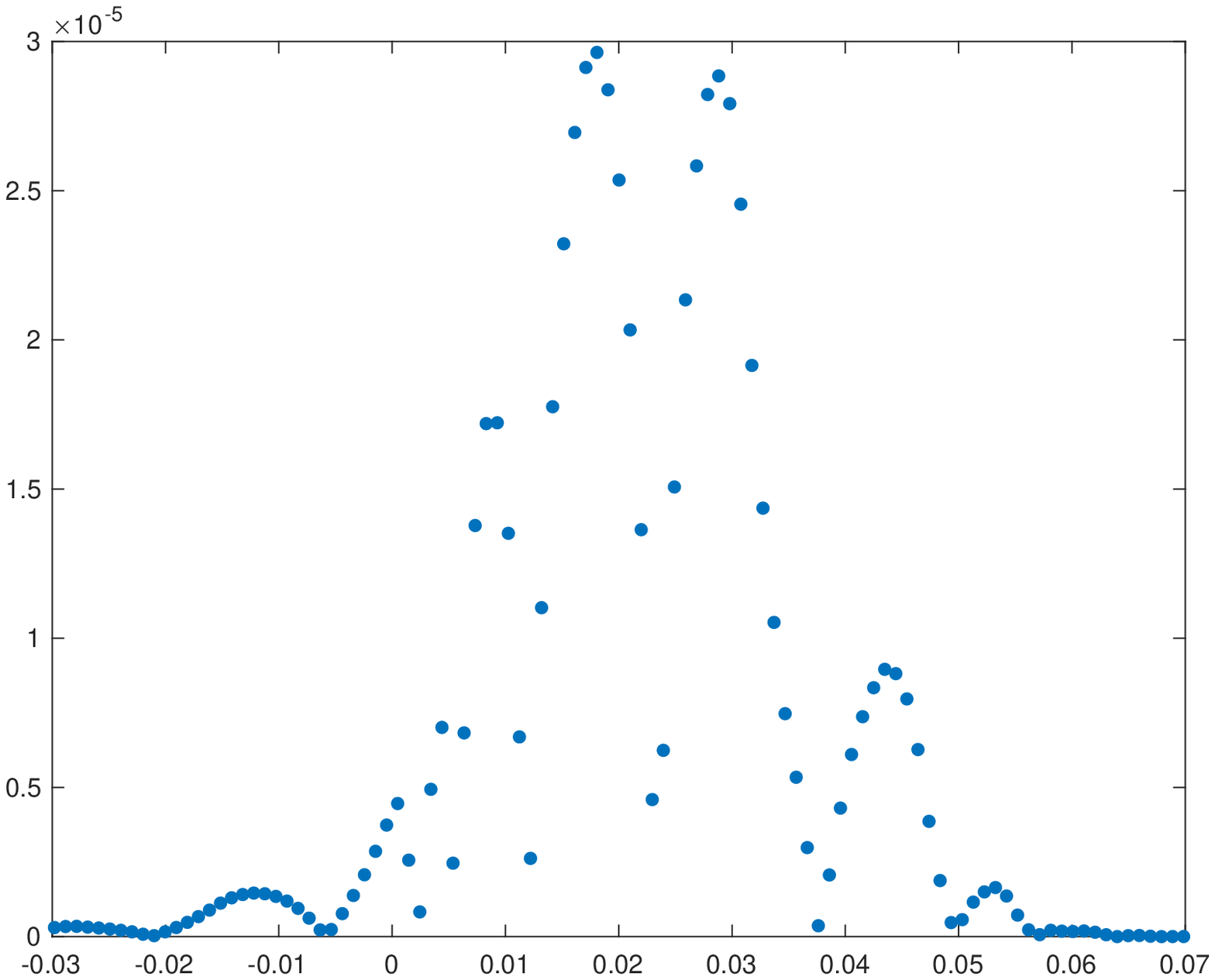,height=6cm}\\
\psfig{figure=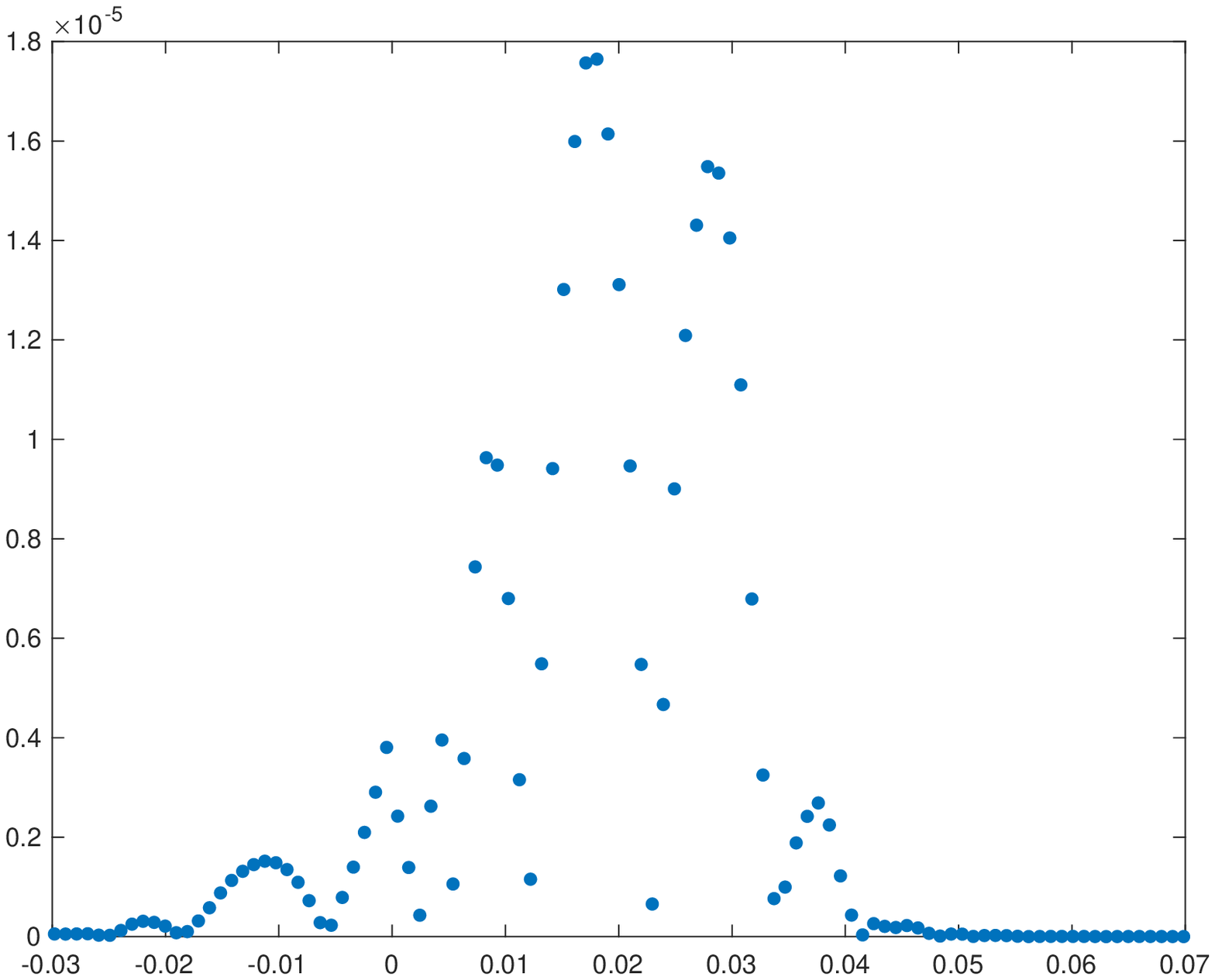,height=6cm}
}
\centerline{\psfig{figure=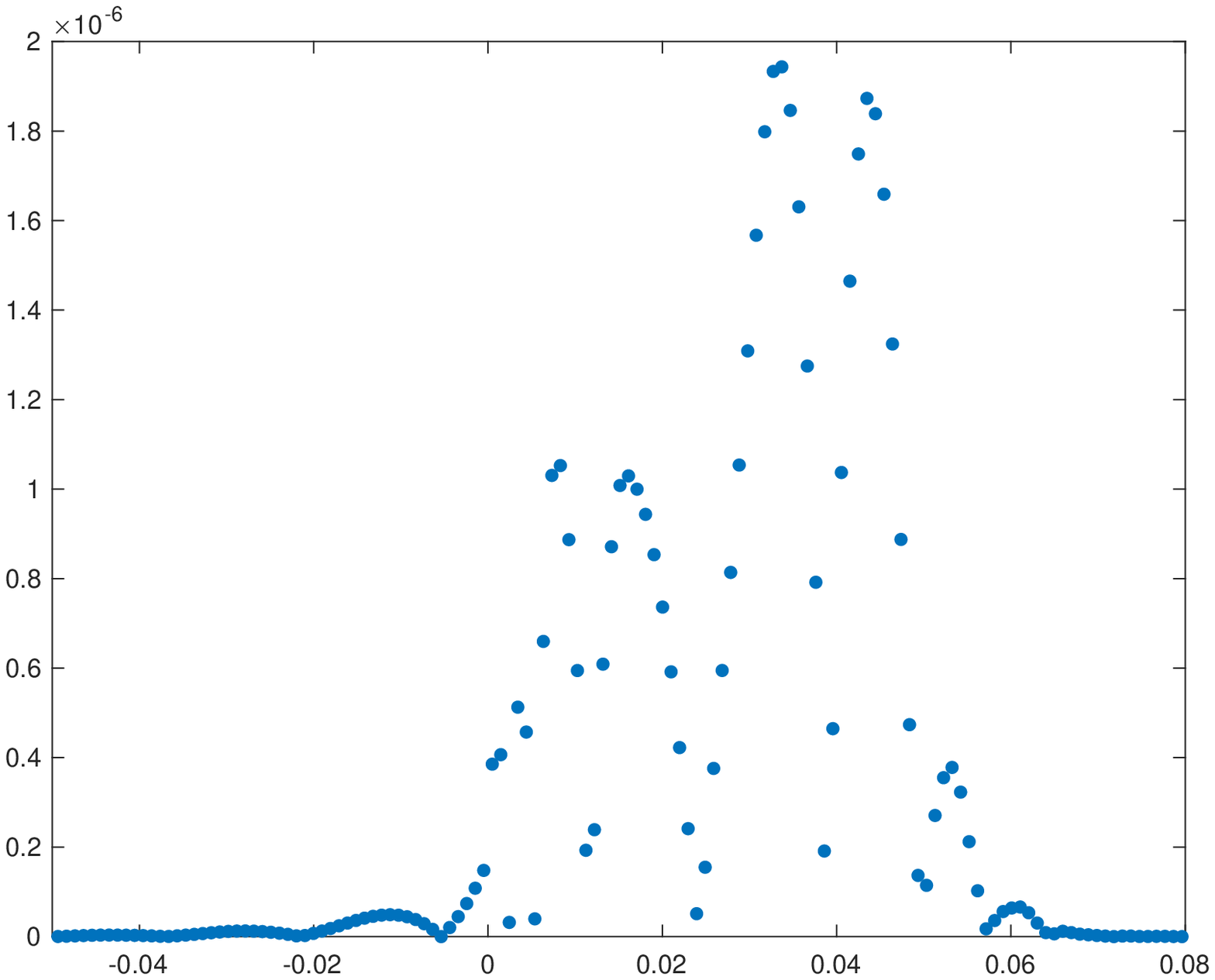,height=6cm}\\
\psfig{figure=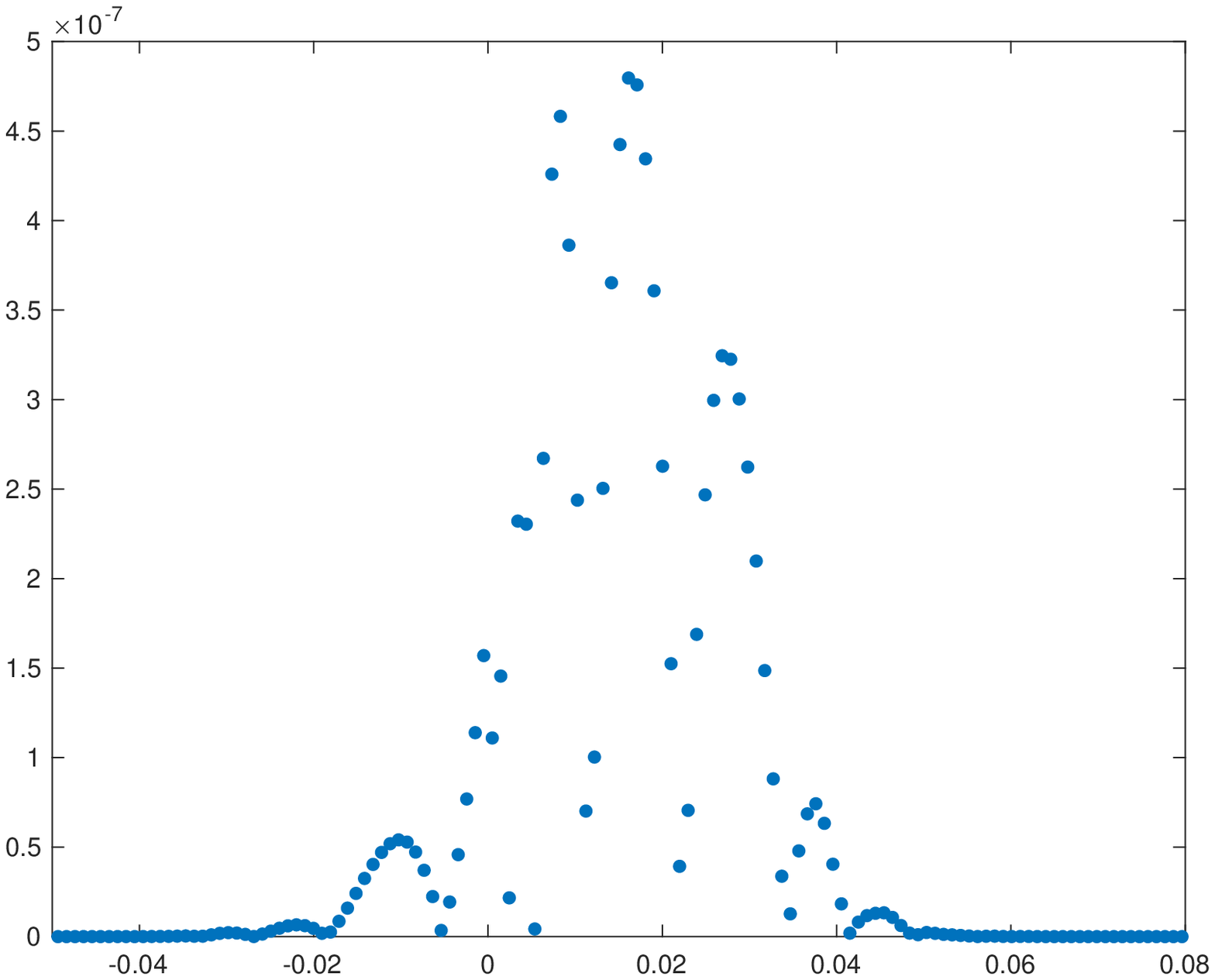,height=6cm}
}
\centerline{\psfig{figure=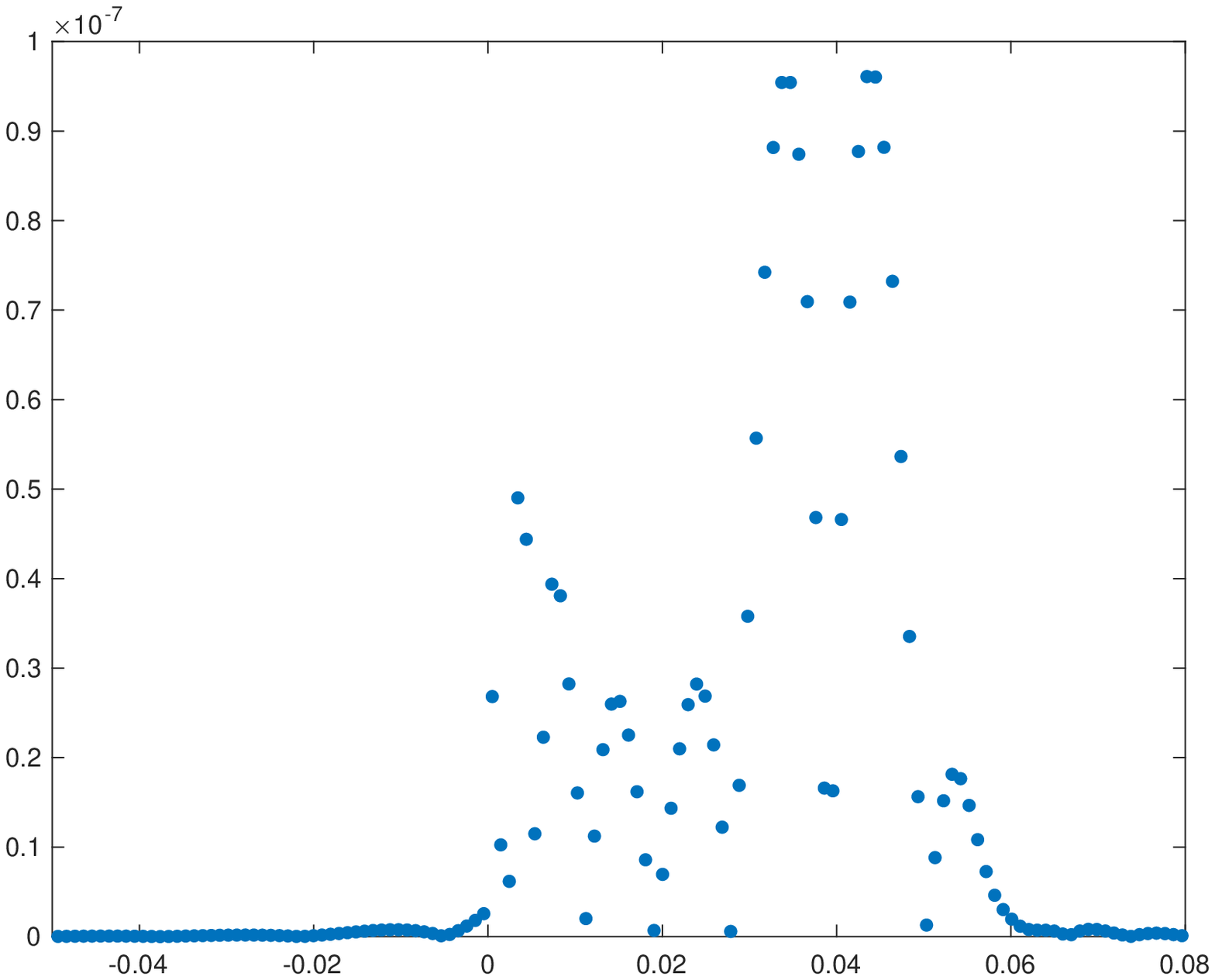,height=6cm}\\
\psfig{figure=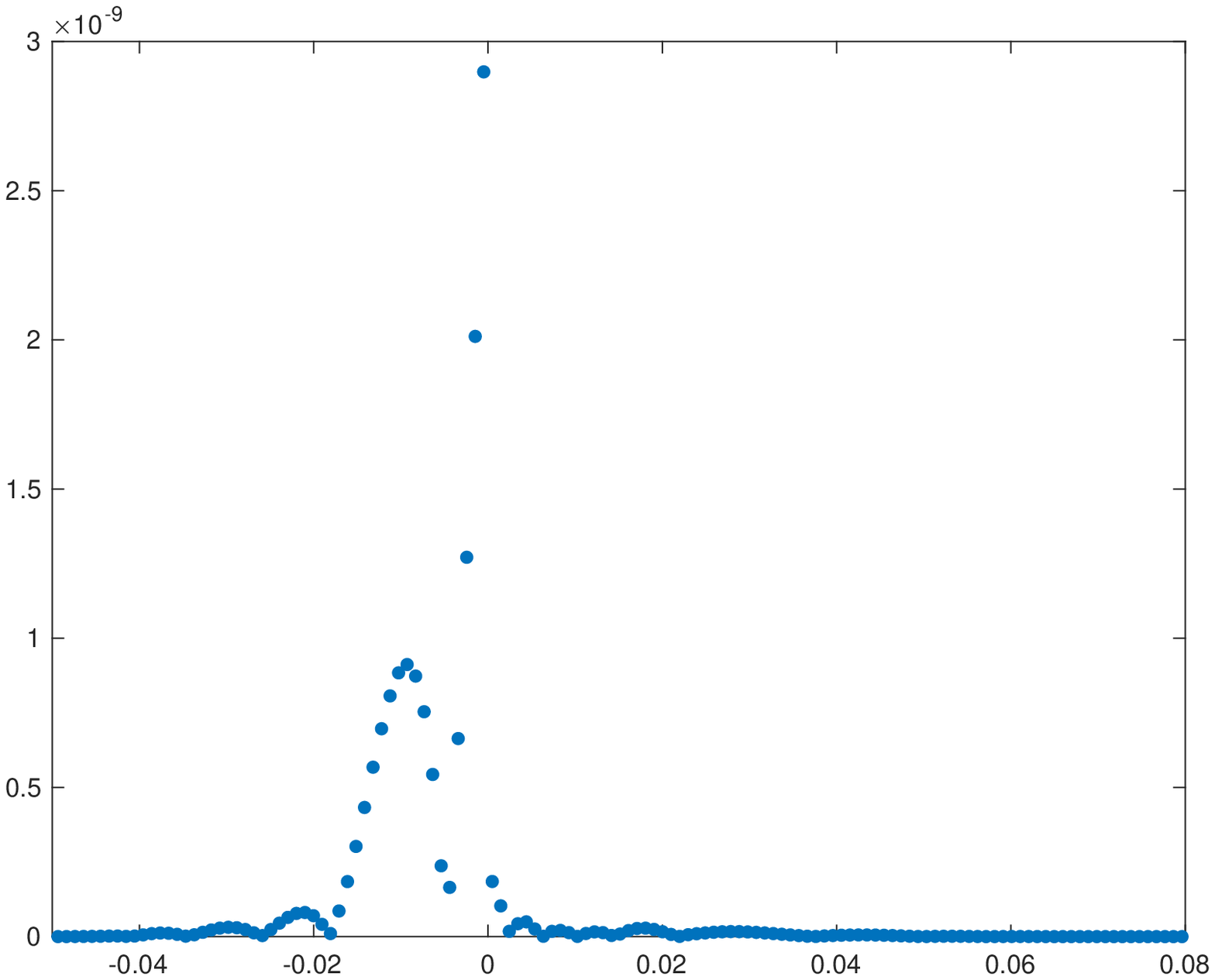,height=6cm}
}

\caption{Error distribution for the function in (\ref{experimento2}) around the discontinuity. The function has been compressed with four scales of multiresolution keeping 4 details (one for scale). }\label{exp2_mult}
\end{figure}

\subsection{Multiresolution of bivariate functions}
The aim of this section is to check how the new and the classical implementation of the WENO-($2r-1$)  algorithm perform in a two-dimensional multiresolution application. We will extend the multiresolution process to two dimensions using tensor product \cite{AD99}. This means that we will use the unidimensional algorithm for processing the rows and the columns of the two-dimensional matrix where the data is stored. Our intention is to check if the new implementation of the WENO-($2r-1$)  algorithm performs better that the classical implementation close to the discontinuities. As it can be seen in previous sections in Tables \ref{tabla_ex1_1} to \ref{tabla_ex2_6}, the size of the error is very small at the cells that are contiguous to the cell that contains the discontinuity, where both algorithms loss their accuracy and attain an error of the size of the discontinuity. Due to this fact, if we want to observe the errors obtained at he mentioned cells, we need to keep one detail for each cell that contains a discontinuity and for each scale of the multiresolution process. We have chosen the bivariate function presented in (\ref{experimento3}), that contains only vertical edges. The reason is that tensor product multiresolution results in great details for oblique or curved edges, that would hide the ones produced directly by any implementation of WENO-($2r-1$)  algorithm close to the discontinuities, as the latter are much smaller. Thus, as we have done for the one dimensional experiments, in order to be able to show the size of the errors close to the discontinuities, we will keep at each scale of multiresolution the details associated to cells that contains a discontinuity. In the next experiment we start from initial data that has $1024\times1024$ cells and we have three vertical discontinuities at every multiresolution scale. Thus, with four multiresolution scales we will keep 2880 details that account for $3\times512$ details for the first step of multiresolution, $3\times256$ for the second, $3\times128$ for the third and $3\times64$ for the fourth.
{\it\bf Example 3} Let's consider the following bivariate function, that has been represented in Figure \ref{f3},
\begin{equation}\label{experimento3}
\scalemath{0.9}{
f(x,y)=\left\{\begin{array}{ll}
3+e^{xy}, & \textrm{ if } -1\leq x<-0.5, -1\leq y<1,\\
3, & \textrm{ if } x=-0.5,\\
1+\cos(2\pi xy), & \textrm{ if } -0.5< x\leq0, -1\leq y<1,\\
1, & \textrm{ if } x=0,\\
\sin(2\pi xy), & \textrm{ if } 0< x\leq.5, -1\leq y<1,\\
-2, & \textrm{ if } x=0.5,\\
-4+\sin(2\pi xy), & \textrm{ if } 0.5< x\leq1, -1\leq y<1.
\end{array}\right.
}
\end{equation}

\begin{figure}[H]
\centerline{\psfig{figure=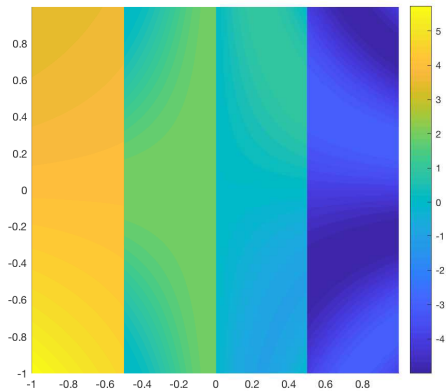,height=6cm}
}
\caption{Bivariate function in (\ref{experimento3}).}\label{f3}
\end{figure}
As we have mentioned before, this function has been initially sampled with $1024\times1024$ cells and we perform 4 levels of multiresolution. The norms of the errors obtained by the classical and the new implementation of the WENO-($2r-1$)  algorithms are shown in Table \ref{tabla_mult_3}. In Figure \ref{figexp3} to the left we show the error distribution that we have obtained using the classical implementation of WENO-($2r-1$)  algorithm and to the right, the error distribution obtained using the new implementation of the WENO-($2r-1$)  algorithm. The conclusion that can be obtained is that, for this application, the new implementation of the WENO-($2r-1$)  algorithm manages to obtain smaller errors around the discontinuity than the classical implementation. We can also observe that, the new implementation compresses the error distribution towards the discontinuity. The computational cost has been obtained performing 100 executions of the algorithms and obtaining the mean of the computational times of all the executions. We can see that the computational cost is similar for both algorithms.
\begin{table}[!ht]
\begin{center}
\resizebox{15cm}{!} {
\begin{tabular}{|c|c|c|c|c|c|c|c|c|c|c|c|c|c|c|}
\hline $i$ &WENO-6 & new WENO-6&WENO-8 & new WENO-8&WENO-10 & new WENO-10
              \\
\hline  $l_{\infty}$&3.4437e-02   &3.1458e-02& 1.0015e-02&7.0379e-03&7.1835e-03&4.1436e-03
            \\
\hline  $l_2$& 1.4241e-06&9.3078e-07&2.5654e-07&1.0417e-07&7.0688e-08&1.4339e-08
            \\
\hline  $l_1$& 2.0936e-04&1.5995e-04&8.6709e-05&4.6845e-05&3.4714e-05&1.2318e-05
            \\
\hline  Comp. time (s)& 4.519&5.026&5.977&6.552&6.653&8.105
            \\
\hline
\end{tabular}
}\caption{Norms of the error and computational time (in seconds) obtained when compressing the function in (\ref{experimento3}) with four scales of multiresolution and keeping $960$ details. The initial resolution is $1024\times1024$ cells.}\label{tabla_mult_3}
\end{center}
\end{table}

\begin{figure}[!ht]
\centerline{\psfig{figure=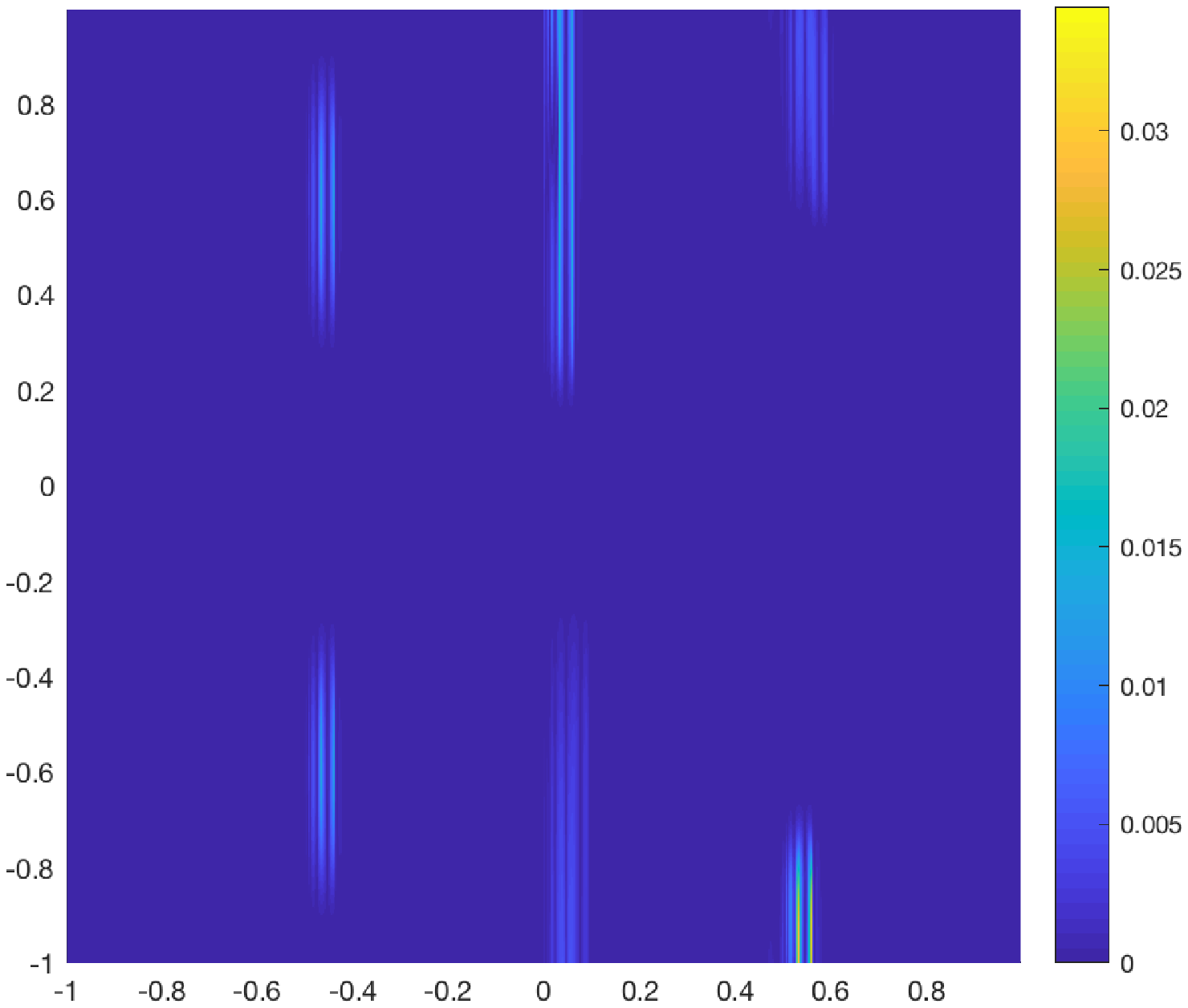,height=6cm}\\
\psfig{figure=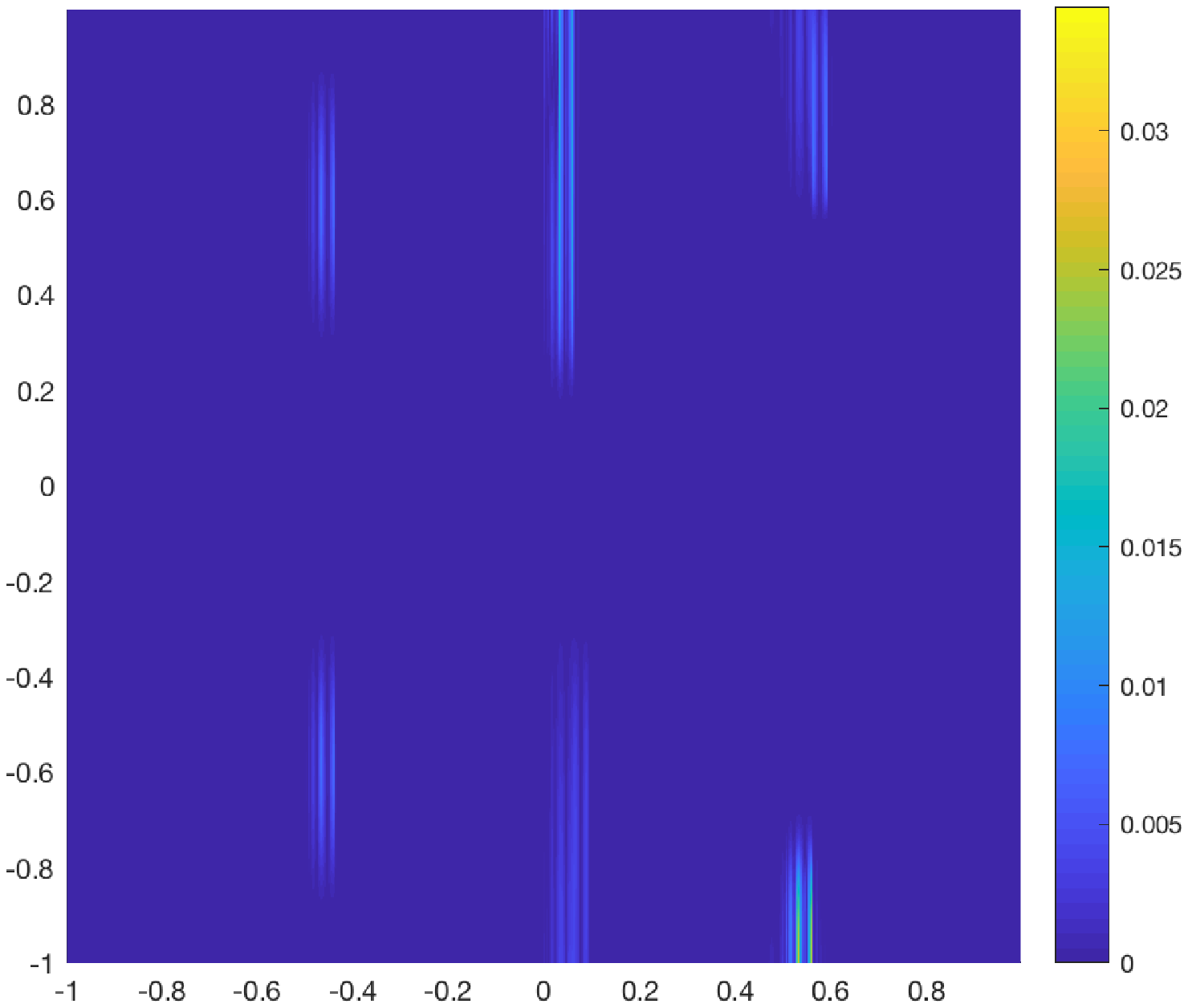,height=6cm}
}
\centerline{\psfig{figure=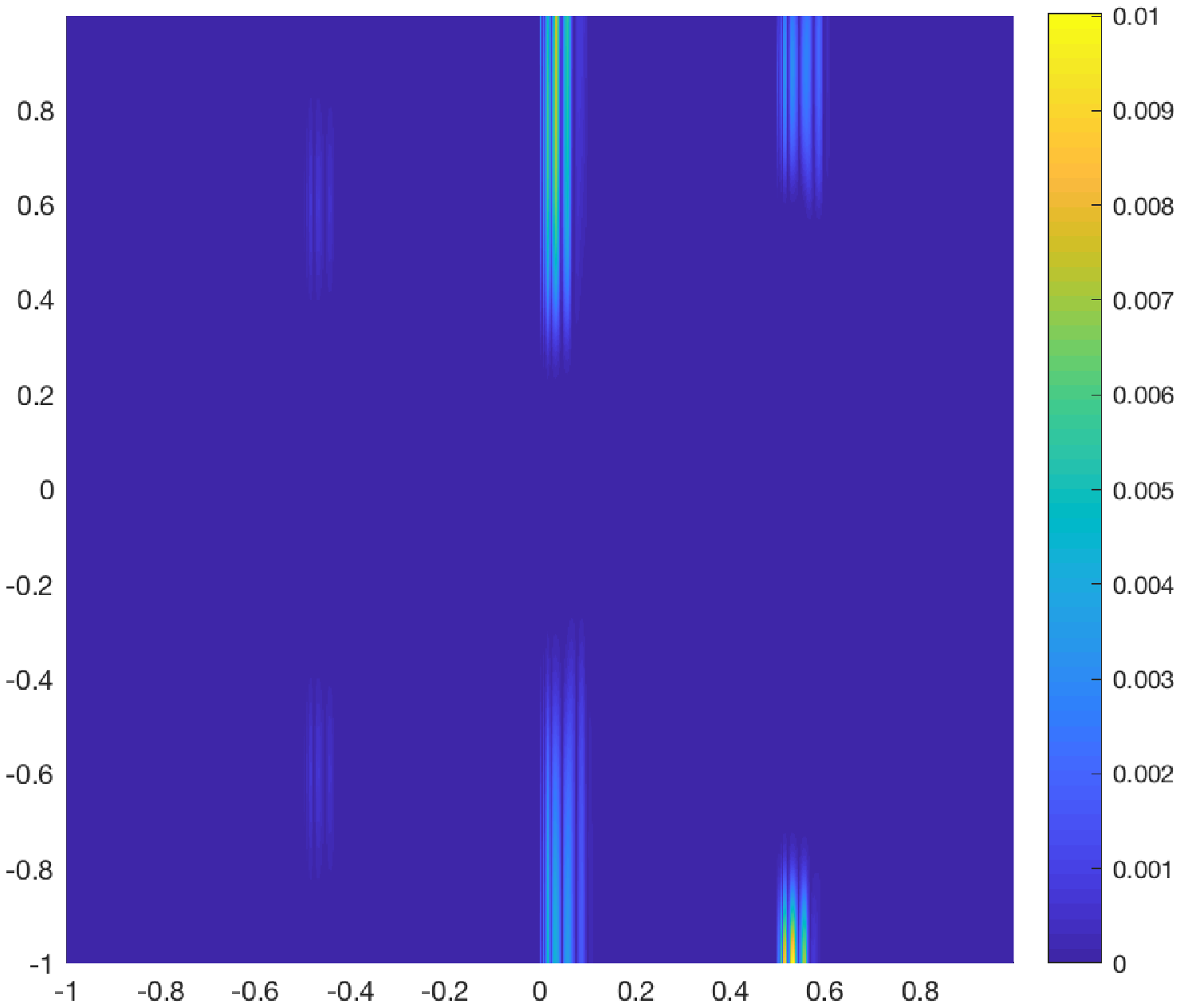,height=6cm}\\
\psfig{figure=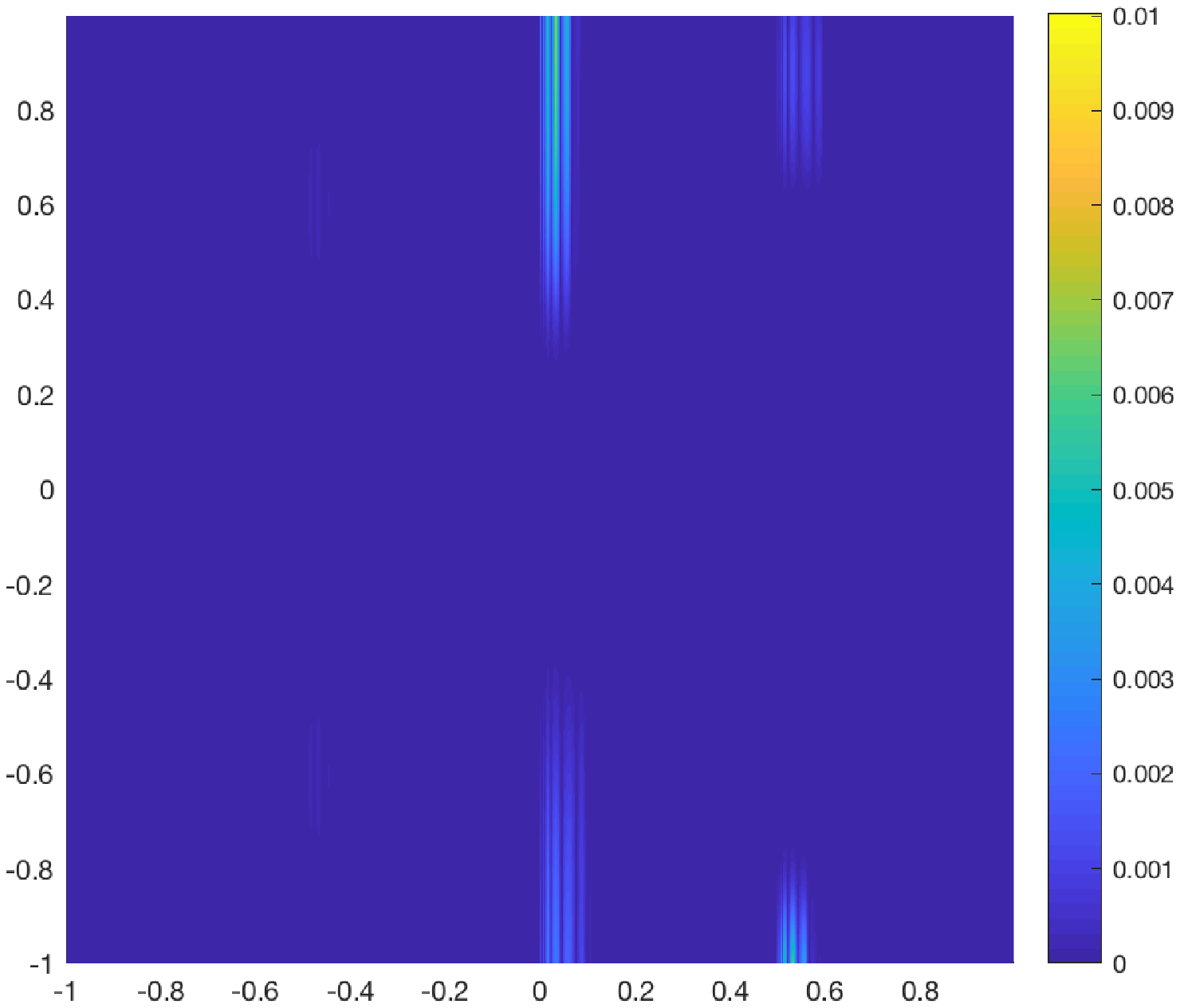,height=6cm}
}
\centerline{\psfig{figure=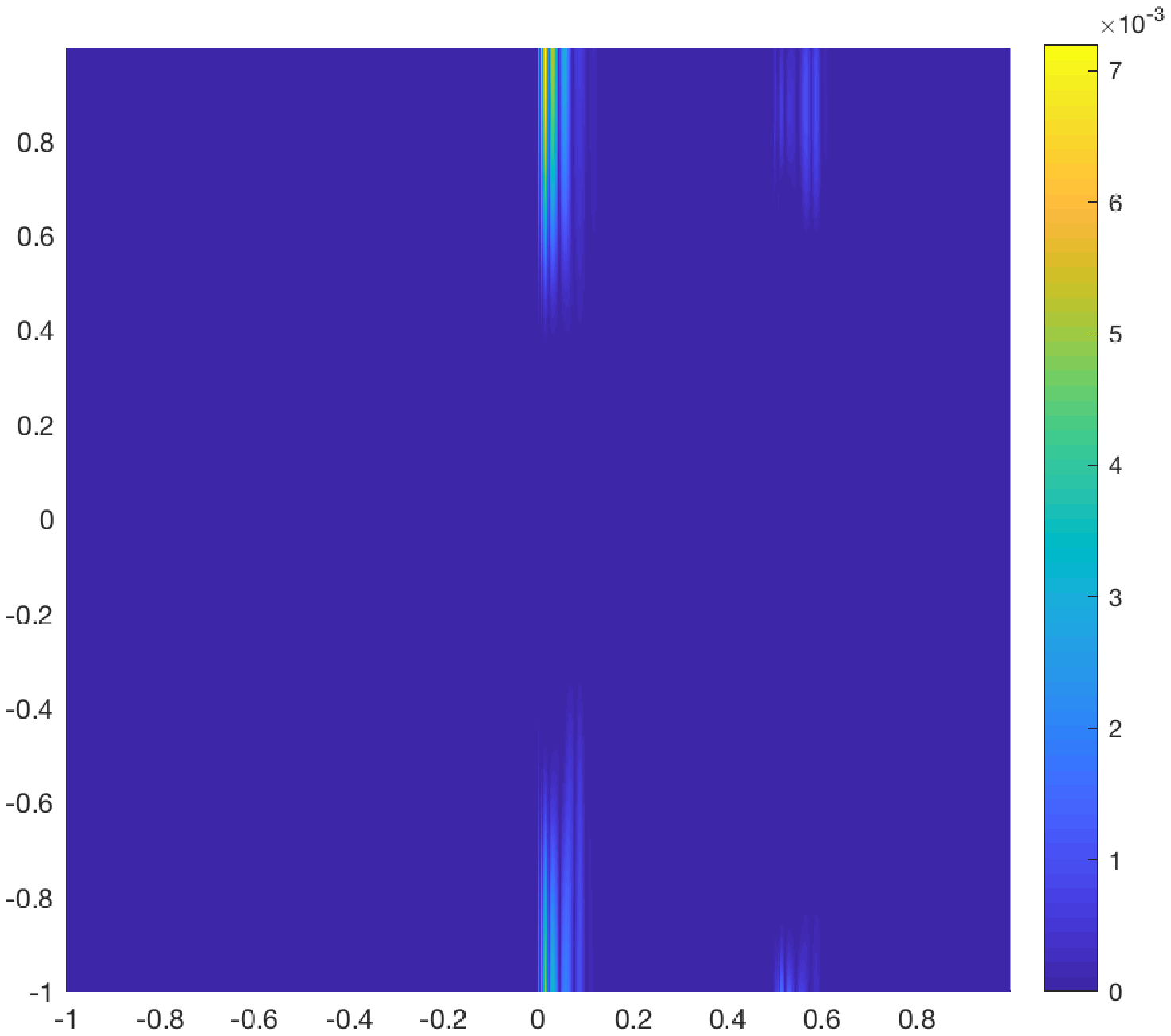,height=6cm}\\
\psfig{figure=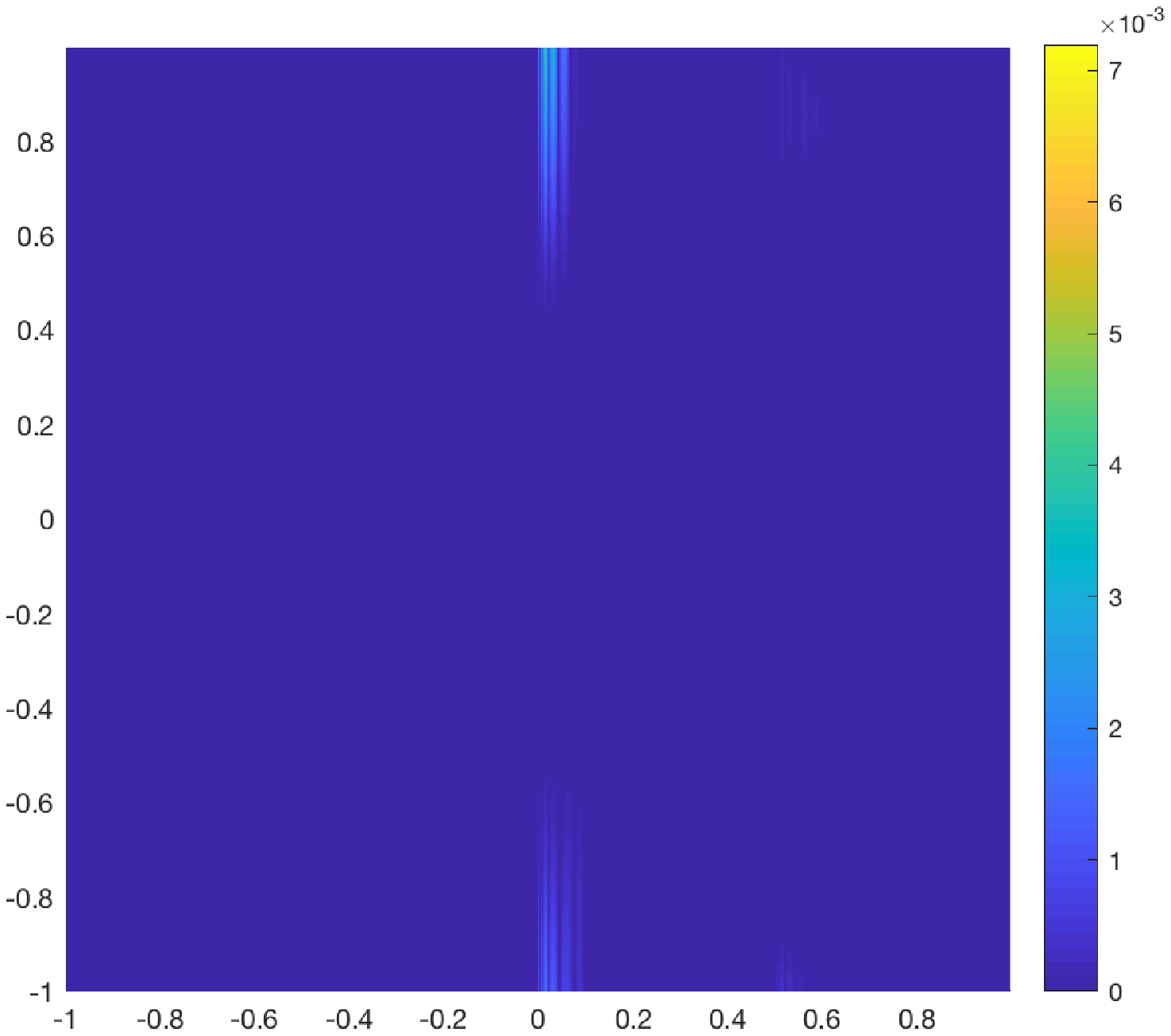,height=6cm}
}
\caption{Distribution of the error obtained when compressing the bivariate function in (\ref{experimento3}) using 4 levels of multiresolution and keeping $960$ details. The initial resolution is $1024\times1024$ cells.}\label{figexp3}
\end{figure}


\section{Conclusions}\label{conclusions}
In \cite{cellwenoamatruizshu,WENO_nuevo,generalizacion} the authors presented a new algorithm that improves the performance of WENO-6 algorithm close to the discontinuities. In \cite{articulopaper} we extended the strategy for any WENO-($2r-1$)  algorithm for data discretized in the point values. In this article we have extended the strategy for data discretized in the cell averages. The new implementation of the WENO-($2r-1$)  algorithm behaves theoretically and numerically in the way designed, so that the order of accuracy decreases progre\-ssively towards the discontinuity.  The numerical experiments presented allow to confirm this fact. Also, we have shown that the new implementation of the algorithm is suitable for multiresolution applications of univariate and bivariate functions, improving the results attained by the classical implementation of WENO-($2r-1$)  algorithm.


%
\bibliographystyle{model5-names}
\bibliography{bibliografia_bibdesk_no_doi_url}
\pagebreak
\end{document}